\documentclass[12pt]{amsart}
\usepackage{amsfonts}
\usepackage{amsfonts,latexsym,rawfonts,amsmath,amssymb,amsthm}
\usepackage[plainpages=false]{hyperref}

\usepackage{graphicx}

\RequirePackage{color}

\numberwithin{equation}{section}

\newcommand{\beq}{\begin{equation}}
\newcommand{\eeq}{\end{equation}}
\newcommand{\beqs}{\begin{eqnarray*}}
\newcommand{\eeqs}{\end{eqnarray*}}
\newcommand{\beqn}{\begin{eqnarray}}
\newcommand{\eeqn}{\end{eqnarray}}
\newcommand{\beqa}{\begin{array}}
\newcommand{\eeqa}{\end{array}}

\newcommand{\p}{\partial}
\newcommand{\eps}{\varepsilon}

\newcommand{\Om}{\Omega}
\newcommand{\pom}{{\p\Om}}
\newcommand{\bom}{{\overline\Om}}

\newtheorem{prop}{Proposition}[section]
\newtheorem{theo}[prop]{Theorem}
\newtheorem{lem}[prop]{Lemma}

\newtheorem{cor}[prop]{Corollary}
\newtheorem{rem}[prop]{Remark}

\title{A potential theory for the $k$-curvature equation}

\author {Qiuyi Dai\ \ \ \ \ \ Xu-jia Wang\ \ \ \ \ \ Bin Zhou}

\address{Qiuyi Dai, College of Mathematics and Computer Science,
Hunan Normal University, Changsha 410081, China.}

\address{Xu-jia Wang, Mathematical Sciences Institute,
The Australian National University, Canberra, ACT 2601, Australia.}

\address{Bin Zhou, School of Mathematical Sciences, Peking
University, Beijing 100871, China; and Mathematical Sciences Institute,
The Australian National University, Canberra, ACT 2601, Australia.}



\thanks {The first author was supported by NNSFC 11271120.
The second author was supported by ARC DP120102718 and ARC FL130100118.
The third author was supported by ARC DECRA and  NNSFC 11101004.}

\email{qiuyidai@aliyun.com, \ \ xu-jia.wang@anu.edu.au, \ \ bin.zhou@anu.edu.au}

\begin{document}
\maketitle

\bibliographystyle{plain}

\begin{abstract}

In this paper, we introduce a potential theory for the $k$-curvature equation,
which can also be seen as a PDE approach to curvature measures.
We assign a measure to  a bounded,
upper semicontinuous function which is strictly subharmonic
with respect to the $k$-curvature operator, and
establish the weak continuity of the measure.

\end{abstract}

\maketitle

\baselineskip=16.4pt
\parskip=3pt

\section{Introduction}


The potential theory for nonlinear elliptic equations has been extensively studied.
For quasilinear elliptic equations, we refer the reader to \cite {HKM}.
For the complex Monge-Amp\`ere equation, the weak convergence for bounded,
monotone pluri-subharmonic functions was established in \cite{BT1, BT2}.
See \cite {Ce1, Ce2, Kl, Kol, BEGZ} for further development. For the $k$-Hessian equations,
analogous but stronger results were obtained in \cite{TW2} (see also \cite {W2}).

In a recent paper \cite {DTW},  the mean curvature measure was
established for upper semicontinuous functions which are subharmonic
with respect to the mean curvature operator,
which is a notion weaker than
the generalized solution studied by Giusti \cite{G1, G2}.
The purpose of this paper is to establish the $k$-curvature measures,
where $1 < k < n$,
for upper semicontinuous functions which are subharmonic
with respect to the $k$-curvature operator, as well as the weak convergence of measures.

For a function $u\in{C^2(\Omega)}$,
where $\Om$ is a bounded domain in $\Bbb R^n$,
the $k$-curvature of the graph of $u$ is the $k^{th}$ elementary
symmetric polynomial of the principal curvatures $(\kappa_1, \cdots, \kappa_n)$
of the graph of $u$. 
It can also be written as
\beq
H_k[u]=\sigma_k\Big(D\big(\frac{Du}{w}\big)\Big),
\eeq
where $w=\sqrt{1+|Du|^2}$, and $\sigma_k$
denotes the sum of the principal minors of order $k$
of the matrix $D\big(\frac{Du}{w}\big)$.
In particular, $H_1$ is the mean curvature and $H_n$ is Gauss curvature.

We say a function $u$ in $C^2(\Omega)$ is
{\it $H_k$-subharmonic} if $H_j[u]\geq 0$ for $j=1,\cdots, k$,
namely if  the principal curvatures
$\lambda=(\lambda_1, \cdots, \lambda_n)$ of the graph of $u$ lies in
$\bar \Gamma_k$, the closure of the convex cone
\beq
\Gamma_k=\{\lambda\in{\mathbb R^n}\ |\  \sigma_j(\lambda)>0, j=1,\cdots, k\}.
\eeq
This $H_k$-subharmonicity is equivalent to the {\it $H_k$-admissibility} in
\cite {CNS}. We use the terminology subharmonicity, instead of admissibility,
to reflect the following extension to nonsmooth function in the viscosity sense,
as in \cite {TW2, DTW}.
Namely an upper semicontinuous function $u:\ \Omega\to [-\infty, \infty)$
is called subharmonic with respect to the $k$-curvature operator $H_k$,
or {\it $H_k$-subharmonic} for short,
if the set $\{u=-\infty\}$ has measure zero and
for any open set $\omega\Subset\Omega$ and any smooth
function $h\in{C^2(\overline\omega)}$ with $H_k[h]\leq 0$,
$h\geq u$ on $\p \omega$, one has $h\geq u$ in $\omega$.
This is equivalent to the inequality $H_k[u]\geq 0$
holding in the viscosity sense \cite{T1}.
We denote the set of all $H_k$-subharmonic functions on $\Omega$ by $\mathcal{SH}_k(\Omega)$.
We say $u\in \mathcal{SH}_k(\Omega)$ is {\it strictly $H_k$-subharmonic} if $h>u$
for any $\omega\Subset\Omega$ and any continuous function $h$
with $H_k[h]=0$ in the viscosity sense, $h\geq u$ on $\partial\omega$.

The main result of the paper is the following

\begin{theo}
(i) For any bounded strictly $H_k$-subharmonic function $u$,
there exists an associated Radon measure $\mu_k[u]$,
such that $H_k[u]$ is the density function of $\mu_k[u]$ if $u\in C^2$.
\newline
(ii) If $\{u_j\}$ is a sequence of bounded strictly $H_k$-subharmonic
functions which converges to a bounded strictly $H_k$-subharmonic $u$ a.e.,
then $\mu_k[u_j]\to \mu_k[u]$ weakly.
\end{theo}

The notion of curvature measures in Theorem 1.1
is closely related to that of Federer's \cite {F}.
The curvature measures in Theorem 1.1
are defined in the domain $\Om$
while Federer's are defined in the space $\Bbb R^{n+1}$,
supported on the graph of the function.
But they are essentially the same as
the curvature measures in Theorem 1.1
are projections of Federer's from $\Bbb R^{n+1}$ to $\Bbb R^n$.
When $k=n$, the $k$-curvature is the Gauss curvature and
the corresponding curvature measure has  been extensively studied \cite {S}.

Federer's curvature measures apply to general sets in the Euclidean space $\Bbb R^{n+1}$
of positive reach. If the boundary of the set is the graph of a function,
the positive reach condition means the function is semi-convex,
which is a condition different from the $H_k$-subharmonicity in Theorem 1.1.

The weak convergence in part (ii) is
a stability of the curvature measures $\mu_k[u]$,
as the convergence is independent of the sequence $\{u_j\}$.
It also allows us to assign a measure
to a bounded strictly $H_k$-subharmonic function $u$.
The stability of curvature measures is useful in the numerical computation
of curvature measures, see \cite {CCD, CCL}.

For the mean curvature operator ($k=1$), this result was established in \cite{DTW}.
In this paper, we will use some ideas from \cite{DTW} and \cite{HKM}, such as
the Perron lifting in \S3. The argument in \S5 was also inspired
by that in \cite{DTW}.
However, new difficulties arise in treating the cases $1<k<n$.
One is that the operator $H_k$ is fully nonlinear and there is no interior regularity
for the equation, as with the case of the complex Monge-Amp\`ere equation.
But the main obstacle is that the set of $H_k$-subharmonic functions is not convex
and so we cannot use the perturbation or mollification techniques.
For example, if $u$ is $H_k$-subharmonic,
 $u+\eps\phi$ may fail to be $H_k$-subharmonic even if $\phi$ is {\it convex}.
An simple example is: letting $u$ be a smooth $H_1[u]$-subharmonic function in $\mathbb R^2$ satisfying $Du(0)=0$, $u_{xx}(0)=-1$, $u_{yy}(0)=1$ and $u_{xy}(0)=0$. Then one can check that $H_1[u+\epsilon y]<0$ at $0$.

In this paper we introduce a deformation technique,
by studying the Dirichlet problem of an associated obstacle problem.
This deformation argument provides an approximation to
a nonsmooth $H_k$-subharmonic function
by smooth $H_k$-subharmonic functions of positive $k$-curvature.
This is a key ingredient in the argument of the paper.
This technique can also be used to give a new proof of the global
regularization of pluri-subharmonic functions on compact K\"ahler manifolds.
Namely a pluri-subharmonic function on a compact K\"ahler manifold
can be approximated by smooth ones (see Appendix 2).
Also because we cannot use the perturbation technique,
we have to assume the strict $H_k$-subharmonicity condition in Theorem 1.1.

The organisation of the paper is as follows.
In \S2, we prove a monotonicity inequality for the $k$-curvature operators.
The monotonicity for $k$-Hessian operators can be obtained
by a simple integration by parts \cite{TW1}.
But for the $k$-curvature operators, the computation is much more complicated.

In \S3 we introduce the Perron lifting.
In \S4 we prove that every $H_k$-subharmonic function
can be approximated locally by a sequence of smooth
$H_k$-subharmonic functions.
The approximation for $H_1$-subharmonic functions
was obtained in \cite{DTW} by Perron lifting.
This technique does not work when $k>1$, due to the lack of interior regularity of
the equation $H_k[u]=0$.
In this paper we prove the smooth approximation
by solving a sequence of  Dirichlet problems with obstacles
and introducing the notion of $H_k$-subharmonic envelope.

In \S5 we prove Theorem 1.1,
by making a perturbation for a sequence
of $H_k$-subharmonic functions in an annulus,
and using the monotonicity formula established in \S2.
To keep the $H_k$-subharmonicity under the perturbation,
we need the assumption of the strict $H_k$-subharmonicity in Theorem 1.1.

In Appendix I we prove the existence of solutions to the Dirichlet problems
with obstacle, which is needed in Section 4.
As an application we prove in Appendix I\!I
that a pluri-subharmonic function on a compact K\"ahler manifold
can be approximated by smooth ones.

\vskip20pt

\section{An integral monotonicity inequality}


Denote $A=D\left(\frac{Du}{w}\right)=(a_{ij})$, where
\beq
a_{ij}=\frac{u_{ij}}{w}-\frac{u_{il}u_lu_j}{w^3}.
\eeq
Following \cite{CNS}, it can be written as
\beq
a_{ij}=\frac{1}{w}u_{ip}b^{pq}b^{qj},
\eeq
where $b^{ij}=\delta_{ij}-\frac{u_iu_j}{w(1+w)}$.
We have \cite {T1}
\beq
H_k[u]=\frac{1}{k!}\sum \Big(\begin{array}{ccc}
     i_1 & \cdots & i_k \\j_1 & \cdots & j_k\end{array}
    \Big)a_{i_1j_1}\cdots a_{i_kj_k}.
\eeq
By proposition 4.1 in \cite{Re}, $H_k$ can be written in the divergence form
\beq
H_k[u]=\frac{1}{k}\Big(H^{ij}\frac{u_j}{w}\Big)_i,
\eeq
where a subscript of a function denotes partial derivative, namely
$u_j=u_{x_j}$, $u_{ij}=u_{x_ix_j}$, and
\beqn
H^{ij} &=&\frac{\p H_k}{\p a_{ij}}\\
        &=&\frac{1}{(k-1)!}\sum \Big(\begin{array}{cccc}
          i_1 & \cdots & i_{k-1} & i \\ j_1 & \cdots & j_{k-1} & j
          \end{array}\Big)a_{i_1j_1}\cdots a_{i_{k-1}j_{k-1}}.\nonumber
   \eeqn
By the divergence theorem,
$$\int_{\Omega}H_k[u]=\int_{\p \Omega}X_u\cdot \gamma,$$
where $\gamma$ is the unit outer normal of $\p \Omega$, and
$X_u=(H^{1j}\frac{u_j}{w}, \cdots, H^{nj}\frac{u_j}{w})$.

\begin{lem}
Suppose $u, v\in\mathcal{SH}_k(\Omega)\cap C^2(\overline\Omega)$.
If $u=v$ and $Du=Dv$ on $\p \Omega$, then
\beq
\int_{\Omega}H_k[u]=\int_{\Omega}H_k[v].
\eeq
\end{lem}

\begin{proof}
First, we show that $X_u\cdot \gamma$ is invariant under orthogonal transformation.
For convenience, by (2.2), we rewrite $H^{ij}$ as
$$H^{ij}=w\frac{\p H_k}{\p u_{ip}}b_{pq}b_{qj},$$
where $b_{ij}=\delta_{ij}+\frac{u_iu_j}{1+w}.$
Let $x=Py$ and denote $\tilde u(y)=u(x)$, where $P=(p_{ij})$ is an orthogonal matrix.
Then
\begin{eqnarray*}
&&\tilde w=w, \\
&&\tilde b_{ij}=\delta_{ij}+\frac{u_\alpha u_\beta p_{\alpha i}p_{\beta j}}{1+w},\\
&&\frac{\p H_k[\tilde u]}{\p \tilde u_{y_iy_j}}=\frac{\p H_k}{\p u_{\alpha\beta}} p_{\alpha i}p_{\beta j}.
\end{eqnarray*}
Note that $\tilde u_{y_i}=u_{x_\alpha}p_{\alpha i}$ and $\tilde \gamma_i=p_{\alpha i}\gamma_\alpha$.
Hence,
\beqs
X_{\tilde u}\cdot\tilde \gamma
 &=&\tilde H^{ij}\frac{\tilde u_j}{\tilde w}\tilde \gamma_i\\
 &=&\frac{\p H_k}{\p u_{\nu\mu}} p_{\nu i}p_{\mu m}
\cdot (\delta_{ml} + \frac{u_\alpha u_\beta p_{\alpha m}p_{\beta l}}{1+w})
    \cdot (\delta_{lj}+\frac{u_\alpha u_\beta p_{\alpha l}p_{\beta j}}{1+w}) \cdot
     u_rp_{rj} p_{s i}\gamma_s\\
 &=&\frac{\p H_k}{\p u_{s\mu}} \cdot (p_{\mu l} + \frac{u_\mu u_\beta p_{\beta l}}{1+w})
    \cdot (p_{lr}+\frac{u_\alpha u_r p_{\alpha l}}{1+w}) \cdot
     u_r\gamma_s\\
&=&\frac{\p H_k}{\p u_{s\mu}} \cdot \left[\delta_{\mu r} +2\frac{u_\mu u_r}{1+w}+\frac{u^2_\alpha u_ru_\mu}{(1+w)^2}\right]    \cdot      u_r\gamma_s\\
&=&\frac{\p H_k}{\p u_{s\mu}}\cdot (\delta_{\mu p}+\frac{u_\mu u_p}{1+w}) 
\cdot (\delta_{pr}+\frac{u_r u_p}{1+w})\cdot  u_r\gamma_s\\
&=& X_u\cdot \gamma.
\eeqs

Now for any given point $p\in \p \Omega$,
by a translation and rotation of the coordinates,
we may assume that $p$ is the origin and
locally $\p \Omega$ is given by $x_n=\rho(x')$ such that the inner normal
of $\p \Omega$ is $(0, \cdots, 0,1)$ at $p$. By $u=v$ on $\p \Omega$, we have
\beq\label{bdy}
u_{ij}+u_n\rho_{ij}=v_{ij}+v_n\rho_{ij},\ \  1\leq i,j\leq n-1.
\eeq
As $Du=Dv$ on $\pom$, we have $u_{ij}=v_{ij}$ and
$u_{in}=v_{in}$ at $p$ for $i, j<n$.
It is clear that 
$$X_u\cdot\gamma=-\sum \frac{\p H_k}{\p u_{n i}}\cdot (\delta_{i p}+\frac{u_i u_p}{1+w}) 
\cdot (\delta_{pj}+\frac{u_j u_p}{1+w})\cdot  u_j$$
 is independent of $u_{nn}$ since $\frac{\p H_k}{\p u_{n i}}$ are all independent of $u_{nn}$.
 Hence $X_u\cdot \gamma= X_v\cdot\gamma$ at $p$. Since $p$ is an arbitrary point on $\p \Omega$,
the lemma follows.
\end{proof}

\begin{lem}
We have
\beqn
\frac{\partial H_k}{\partial u_{nn}}
 &=&\frac{1}{k}\frac{\bar w^{k+1}}{w^{k+2}}\sigma_{k-1}(\bar A)\\
 &=& \frac{1}{k!}\frac{\bar w^{k+1}}{w^{k+2}}\sum \left(\begin{array}{ccc}
     i_1 & \cdot\cdot\cdot & i_{k-1} \\
     j_1 & \cdot\cdot\cdot & j_{k-1}\end{array}\right)
     \bar a_{i_1j_1}\cdot\cdot\cdot \bar a_{i_{k-1}j_{k-1}},\nonumber
\eeqn
where $\bar A=(\bar a_{ij})$,
\beqs
\bar a_{ij}
  &=&\frac{u_{ij}}{\bar w}-\frac{\sum_{l=1}^{n-1}u_{il}u_lu_j}{\bar w^3}, \ 1\leq i, j\leq n-1,\\
\bar w &=& \sqrt{1+{\sum}_{i=1}^{n-1}u_i^2}.
\eeqs
\end{lem}

\begin{proof}
By definition,
\begin{eqnarray*}
\frac{\partial H_k}{\partial u_{nn}}=
\frac{\partial H_k}{\partial a_{nn}}\cdot\left(\frac{1}{w}-\frac{u_n^2}{w^3}\right)+\sum_{j=1}^{n-1}\frac{\partial H_k}{\partial a_{nj}}\cdot\left(-\frac{u_nu_{j}}{w^3}\right).
\end{eqnarray*}
Since $\frac{\partial H_k}{\partial a_{nn}}$ depends only  on $\{a_{il}\}_{i, l\neq n}$
and $\frac{\partial H_k}{\partial a_{nj}}$ depends only on $\{a_{il}\}_{i\neq n, l\neq j}$,
it is easy to see that $\frac{\partial H_k}{\partial u_{nn}}$ is independent of $u_{nn}$.
We will further show that $\frac{\partial H_k}{\partial u_{nn}}$
is independent of $u_{ni}$ for any $1\leq i\leq n-1$ and
$w^{k+2}\frac{\partial H_k}{\partial u_{nn}}$ is independent of $u_n$.

First, for any $1\leq i\leq n-1$,
\begin{eqnarray*}
\frac{\partial^2 H_k}{\partial u_{nn}\partial u_{ni}}
&=&\sum_{j=1}^{n-1} \frac{\partial^2 H_k}{\partial a_{nn}\partial a_{ij}}\cdot\left(-\frac{u_nu_{j}}{w^3}\right) \cdot\left(\frac{1}{w}-\frac{u_n^2}{w^3}\right)\\
&&+\sum_{j,l=1}^{n-1}\frac{\partial^2 H_k}{\partial a_{nj}\partial a_{il}}\cdot\left(-\frac{u_nu_{j}}{w^3}\right) \cdot\left(-\frac{u_nu_{l}}{w^3}\right)\\
&&+\sum_{j=1}^{n-1}\frac{\partial^2 H_k}{\partial a_{nj}\partial a_{in}}\cdot\left(-\frac{u_nu_{j}}{w^3}\right)\cdot\left(\frac{1}{w}-\frac{u_n^2}{w^3}\right)\\
&=& 0.
\end{eqnarray*}
Note that in the last equality we use the fact
$$\frac{\partial^2 H_k}{\partial a_{ip}\partial a_{jq}}=-\frac{\partial^2 H_k}{\partial a_{iq}\partial a_{jp}}, \
1\leq i, j, p, q\leq n.$$

Next, we consider the dependence on $u_n$.
For convenience, we denote
$$\tilde a_{ij}=wa_{ij}=u_{ij}-\frac{u_{il}u_lu_j}{w^2}, \
\tilde H_k=\sigma_k(\{\tilde a_{ij}\})=w^kH_k.$$
Then it follows
$$\frac{\partial \tilde a_{ij}}{\partial u_n}=\begin{cases}
    0,  & \text{$i=n$}, \\[3pt]
    \frac{2u_n^2-w^2}{w^4}\sum_{l=1}^{n-1}u_{il}u_l, & \text{$i<n, j=n$,}\\[3pt]
    -\frac{2}{w^4}\sum_{l=1}^{n-1}u_{il}u_lu_ju_n,  & \text{$i, j<n$},
\end{cases}$$
and
$$w^{k+2}\frac{\partial H_k}{\partial u_{nn}}=
\frac{\partial \tilde H_k}{\partial \tilde a_{nn}}\bar w^2+\sum_{j=1}^{n-1}\frac{\partial \tilde H_k}{\partial \tilde a_{nj}}\cdot (-u_nu_{j}).$$
Since $\frac{\partial H_k}{\partial u_{nn}}$ is independent of $u_{ni}$,
we may assume $u_{ni}=0$.
Hence
\begin{eqnarray*}
\frac{\partial }{\partial u_n}\left(w^{k+2}\frac{\partial H_k}{\partial u_{nn}}\right)&=&
\sum_{i, j, l=1}^{n-1}\frac{\partial^2 \tilde H_k}{\partial \tilde a_{nn}\partial \tilde a_{ij}}\bar w^2\left( -\frac{2}{w^4}\sum_{l=1}^{n-1}u_{il}u_lu_ju_n \right)+\sum_{j=1}^{n-1}\frac{\partial \tilde H_k}{\partial \tilde a_{nj}}\cdot (-u_{j})\\&&+\sum_{i, j=1}^{n-1}\frac{\partial^2 \tilde H_k}{\partial \tilde a_{nj}\partial \tilde a_{in}}\cdot \left(\frac{2u_n^2-w^2}{w^4}\sum_{l=1}^{n-1}u_{il}u_l\right)\cdot (-u_nu_{j})\\
&& +\sum_{i, j=1}^{n-1}\sum_{s=1}^{n-1}\frac{\partial^2 \tilde H_k}{\partial \tilde a_{nj}\partial\tilde a_{is}}\cdot\left( -\frac{2}{w^4}\sum_{l=1}^{n-1}u_{il}u_lu_su_n \right)\cdot (-u_nu_{j})
\end{eqnarray*}
Again by the fact
$$\frac{\partial^2 \tilde H_k}{\partial \tilde a_{ip}\partial \tilde a_{jq}}=-\frac{\partial^2 \tilde H_k}{\partial \tilde a_{iq}\partial \tilde a_{jp}}, \  1\leq i, j, p, q\leq n,$$
we have
\begin{eqnarray*}
\frac{\partial }{\partial u_n}\left(w^{k+2}\frac{\partial H_k}{\partial u_{nn}}\right)&=&
\sum_{j=1}^{n-1}\frac{\partial \tilde H_k}{\partial \tilde a_{nj}}\cdot (-u_{j})+\sum_{i, j=1}^{n-1}\frac{\partial^2 \tilde H_k}{\partial \tilde a_{nj}\partial \tilde a_{in}}\cdot \left(\frac{1}{w^2}
\sum_{l=1}^{n-1}u_{il}u_l\right)\cdot (-u_nu_{j})\\
&=&\sum_{j=1}^{n-1}\frac{\partial \tilde H_k}{\partial \tilde a_{nj}}\cdot (-u_{j})+\sum_{i, j=1}^{n-1}\frac{\partial^2 \tilde H_k}{\partial \tilde a_{nj}\partial \tilde a_{in}}\cdot (-\tilde a_{in})\cdot (-u_{j})\\
&=& 0.
\end{eqnarray*}
This implies that we may further assume $u_n=0$ and the lemma follows by
substituting $a_{ij}=\frac{\bar w}{w}\bar a_{ij}$ into $\frac{\partial H_k}{\partial u_{nn}}=
\frac{\partial H_k}{\partial a_{nn}}w^{-1}$.
\end{proof}

The following monotonicity integral inequality is the main result of this section.
It is critical for the proof of weak continuity in Theorem 1.1.
Note that in the following we do not assume the boundary $\pom$ is $(k-1)$-convex.
Due to the lack of convexity of the set of $H_k$-subharmonic functions,
the computation is quite complicated.

 \begin{lem}
Suppose $u, v\in C^\infty(\overline\Omega)$ are $H_k$-subharmonic functions.
If $u=v$ and $u_\gamma>v_\gamma$ on $\pom$,
then
\beq
\int_{\Omega}H_k[u]\geq \int_{\Omega}H_k[v],
\eeq
where $\gamma$ is the unit outer normal of $\pom$.
\end{lem}

\begin{proof}
By the divergence structure of the $k$-curvature operator $H_k$ \cite {I1, I2},
the integral $\int_\Om H_k[u]$ depends only on the value of $u$ near $\pom$.
Hence by the condition $u=v$ and $u_\gamma>v_\gamma$ on $\pom$,
we may also assume that  $u<v$ in $\Omega$.
This condition also implies that $\pom$ is smooth. Since the proof  is complicated,
we divide it into three steps:

(i) For any $p\in \partial\Omega$ and $\delta>0$,
denote by $$L_\delta(p)=\{p+t\gamma \ |\  -\delta\leq t\leq 0\},$$
where $\gamma$ is the unit outer normal.
Assume $\delta$ is small enough such that the set $\{p_\delta=p-\delta \gamma(p)\ |\ p\in\pom\}$
encloses a subdomain $\Omega_\delta=\{x\in\Omega: d(x,\partial\Omega)>\delta\}\subset\Omega$, i.e.,
$$\Omega\setminus \Omega_\delta=\bigcup_{p\in\partial\Omega}L_\delta(p).$$
We hope to construct a suitable function
$\tilde u=u+\eta \in{C^2(\overline{\Omega\setminus \Omega_{\delta}})}$,
such that for some small constant $c_\delta>0$,
\begin{eqnarray*}
&& \tilde u=v-c_\delta,  \nabla \tilde u=\nabla v     \text{\ \ on $\partial\Omega_\delta$}, \\
&&  \tilde u=u,  \nabla \tilde u=\nabla u      \text{\ \ on $\partial\Omega$}.
\end{eqnarray*}
Indeed we will get a  sequence of functions that approximately satisfy the above condition, see (2.12), (2.13).

By the smoothness of $u, v$, for any given $\eps>0$, we have
$$|u_\gamma(x)-u_\gamma(p)|, \ |v_\gamma(x)-v_\gamma(p)|\leq \eps
\ \forall\  x\in{L_\delta(p)} $$
provided $\delta$ is  small enough. Hence
$$|v(p_\delta)-u(p_\delta)-(u_\gamma(p_\delta)-v_\gamma(p_\delta))\delta|
\leq C\eps\delta$$
on $\p \Omega_{\delta}$. Let
$\lambda=\inf_{\p \Omega}(u_\gamma-v_\gamma)>0$.
For any $p_\delta\in \p \Omega_{\delta}$, denote
\beq
a=v(p_\delta)-u(p_\delta)-\frac{\lambda\delta}{4}>0,
\ \ b=u_\gamma(p_\delta)-v_\gamma(p_\delta)>0.
\eeq
So we have $v-\frac{\lambda\delta}{4}>u$ on $\partial\Omega_\delta$
and $v-\frac{\lambda\delta}{4}<u$ on $\partial\Omega$.
For each $p\in\partial\Omega$, we aim to construct a function to connect
$(p_\delta, v(p_\delta)-\frac{\lambda\delta}{4})$ and $(p, u(p))$, and  whose
first derivative coincides with $v_\gamma(p_\delta)$, $u_\gamma(p)$, respectively, and the
second derivative is  $O(\delta^{-1})$ as $\delta\to 0$ for our purpose. A function behaving like $c_1\frac{(-t)^{1+\alpha}}{\delta^\alpha}+c_2\frac{t^2}{\delta}$($0<\alpha<<1$) with $c_1$, $c_2$ to be determined
would have an $O(1)$ jump for the first derivative and an $O(\delta^{-1})$ jump for the second derivative.

Choose $0<\alpha<< 1$ and $\kappa<<\delta$ such that
$$(b-\frac{1}{4}\lambda)\delta<b\frac{\kappa+\delta}{1+\alpha}.$$
Define a function
\beq
\eta=\eta_{\delta,\kappa}=
s\frac{b}{(1+\alpha)(\kappa+\delta)^\alpha}(-t+\kappa)^{1+\alpha}+(1-s)\frac{b}{2\delta}t^2, \ -\delta\leq t\leq 0
\eeq
on each line segment $L_\delta(p)$, where
$$s=\frac{a-\frac{b\delta}{2}}{\frac{\kappa+\delta}{1+\alpha}b-\frac{b\delta}{2}}>0.$$
Here the small perturbation by constant $\kappa$ ensures $\eta$ is smooth up to $\partial \Omega_\delta$.
Note that $\Omega\setminus \Omega_\delta=\cup_{p\in\partial\Omega}L_\delta(p)$.
We obtain a function
$\tilde u=u+\eta\in{C^2(\overline{\Omega\setminus \Omega_{\delta}})}$.
Combining with (2.10), one can check that
\begin{eqnarray}
&&\tilde u=v-\frac{\lambda\delta}{4}, D\tilde u=Dv \text {\ \ on $\partial \Omega_{\delta}$};\\
&&\tilde u=u+\frac{a-\frac{b\delta}{2}}{\kappa+\delta-\frac{(1+\alpha)\delta}{2}}\frac{\kappa^{1+\alpha}}{(\kappa+\delta)^\alpha}, \tilde u_\gamma=u_\gamma-\frac{a-\frac{b\delta}{2}}{\kappa+\delta-\frac{(1+\alpha)\delta}{2}}\frac{(1+\alpha)\kappa^\alpha}{(\kappa+\delta)^\alpha} \text{\ \ on $\partial \Omega$}.
\end{eqnarray}
By the divergence theorem,
\beq
\int_{\partial \Omega}X_{\tilde u}\cdot\gamma=\int_{\partial \Omega_\delta}X_{\tilde u}\cdot\gamma+
\int_{\Omega\setminus \Omega_{\delta}}H_k[\tilde u].
\eeq
By Lemma 2.1, the right hand side is
$$\int_{\Omega_{\delta}}H_k[v]+\int_{\Omega\setminus \Omega_{\delta}}H_k[\tilde u]$$
while the left hand side converges to
$$\int_{\Omega}H_k[u]$$
when $\kappa\to 0$. It suffices to estimate the second term on the right.

\vskip 20pt

(ii) In this step we estimate the derivatives of $\eta$ in $\Omega\setminus \Omega_\delta$.

For any $x_0\in\Omega\setminus\Omega_\delta$, there exists $p$ on $\partial\Omega$
such that $x_0\in L_\delta(p)$, i.e., $x_0=p+t\gamma(p)$.
By a translation and a rotation of the coordinates,
we may assume that $p=0$, $x_0=(0,..., 0, -\delta)$ and $\gamma(p)=(0,\cdots, 0, 1)$. Denote
$\beta'=D_\gamma u(p)$, $\beta=D_\gamma v(p)$.
Then near $0$ we have
\begin{eqnarray*}
u &=&\sum_{i=1}^{n-1}\alpha_ix_i+\beta' x_n
                +\sum_{i, j=1}^n\frac{A_{ij}}{2}x_ix_j+O(|x|^3),\\
v&=&\sum_{i=1}^{n-1}\alpha_ix_i+\beta x_n
                +\sum_{i, j=1}^n\frac{B_{ij}}{2}x_ix_j+O(|x|^3).
\end{eqnarray*}
By the assumption of the lemma, $\beta'>\beta$ and near $0$, by \eqref{bdy}, $\p \Omega$ is given by
\beq
x_{n}=\rho(x')=\sum_{i, j=1}^{n-1}\frac{B_{ij}-A_{ij}}{2(\beta'-\beta)}x_ix_j+O(|x'|^3),
\eeq
where $x'=(x_1, \cdots, x_{n-1})$. Then,
\beq
\rho_i=\frac{B_{ij}-A_{ij}}{\beta'-\beta}x_j+O(|x'|^2).
\eeq
By (2.11), to estimate the derivatives of $\eta$,
we need to analyse $t$ with respect to the nearby points of $x_0$.
For any $x=(x',x_n)$ near $x_0$, let $p_x=(y', \rho(y'))\in\partial\Omega$ be the point such that $x\in{L_\delta(p_x)}$. From
$$-(\frac{D\rho}{\sqrt{1+|D\rho|^2}}, \frac{-1}{\sqrt{1+|D\rho|^2}})\cdot t+(y',\rho(y'))=(x',x_n),$$
and (2.16), we have
\begin{eqnarray}
&&t=-(\rho(y')-x_n)\sqrt{1+|D\rho|^2},\\[6pt]
&&y_i-\left[\frac{B_{ij}-A_{ij}}{\beta'-\beta}y_j+O(|y'|^2)\right]\cdot
\left[\frac{B_{ij}-A_{ij}}{2(\beta'-\beta)}y_iy_j-x_n+O(|y'|^3)\right]\nonumber\\
&& \ \ \ \ \ \ \ \ \ \ =x_i, 1\leq i\leq n-1.
\end{eqnarray}
By (2.18),
$$x_i=y_i+x_n\frac{B_{ij}-A_{ij}}{\beta'-\beta}x_j+O(|y'|^2).$$
It follows
\beq
y_i=x_i-x_n\frac{B_{ij}-A_{ij}}{\beta'-\beta}x_j+O(|x'|^2).
\eeq
Substituting it into $t$, we have
\begin{eqnarray*}
t&=&-\left[\frac{B_{ij}-A_{ij}}{2(\beta'-\beta)}\left(x_i-x_n\frac{B_{ik}-A_{ik}}{\beta'-\beta}x_k\right)\left(x_j-x_n\frac{B_{jl}-A_{jl}}{\beta'-\beta}x_l\right)+O(|x'|^3)-x_n\right]\\
&&\cdot \left[1+\left(\frac{B_{ij}-A_{ij}}{\beta'-\beta}x_j+O(|x'|^2)\right)^2\right]^{\frac{1}{2}}\\
&=&-\left[\frac{B_{ij}-A_{ij}}{2(\beta'-\beta)}x_ix_j -\left(\frac{B_{ik}-A_{ik}}{\beta'-\beta}\frac{B_{kj}-A_{kj}}{\beta'-\beta}x_n+O(|x_n^2|)\right)x_ix_j+O(|x'|^3)-x_n\right]\\
&&\cdot \left[1+\frac{1}{2}\frac{B_{ik}-A_{ik}}{\beta'-\beta}\frac{B_{il}-A_{il}}{\beta'\beta}x_kx_l +O(|x'|^3)\right]\\
&=&x_n+\left[-\frac{B_{ij}-A_{ij}}{2(\beta'-\beta)}+\frac{3}{2}\frac{B_{ik}-A_{ik}}{\beta'-\beta}\frac{B_{kj}-A_{kj}}{\beta'-\beta}x_n+O(|x_n^2|)\right]x_ix_j+O(|x'|^3).
\end{eqnarray*}
At $x_0$, by (2.11),
\begin{eqnarray}
&&|\eta_i|\leq C\delta, \ |\eta_{ni}|\leq C, \ 1\leq i\leq n-1,\\[6pt]
&&|\eta_n|\leq C,\  \eta_{nn}=s\frac{b(-t+\kappa)^{\alpha-1}}{(\delta+\kappa)^\alpha}+(1-s)\frac{b}{\delta}\geq
C\delta^{-1}.
\end{eqnarray}
where $C$ is a constant depending on $u$, $v$ but independent of $\delta$.
Since $\frac{\partial t}{\partial x_i}=0$ at $x_0$ for $1\leq i\leq n-1$, by the expansion of $t$
\beq
\eta_{ij}=b\left[s\frac{B_{ij}-A_{ij}}{\beta'-\beta} \cdot(\frac{-t+\kappa}{\delta+\kappa})^\alpha
                   -\frac{t}{\delta}\frac{B_{ij}-A_{ij}}{\beta'-\beta}(1-s)\right]+O(\delta).
\eeq
Note that $b\to \beta'-\beta$ as $\delta\to 0$. By (2.22),
\begin{eqnarray}
\tilde u_{ij}
    &=&  u_{ij}+\eta_{ij}\nonumber\\
    &=&  [1-(\frac{-t+\kappa}{\delta+\kappa})^\alpha s-\frac{t}{\delta}(1-s)]A_{ij}
       +[(\frac{-t+\kappa}{\delta+\kappa})^\alpha s+\frac{t}{\delta}(1-s)]B_{ij}+O(\delta)
\end{eqnarray}
for $1\leq i, j\leq n-1$ as $\delta\to 0$.

\vskip 20pt

(iii) In the final step, we estimate the integral
$$\int_{\Omega\setminus \Omega_{\delta}}H_k[\tilde u].$$

For any $x_0\in\Omega\setminus\Omega_\delta$ and
$x_0\in L_\delta(p)$ for some $p$ on $\partial\Omega$. As in step (ii),
we may assume that $p=0$ and $\gamma(p)=(0,\cdots, 0, 1)$.
We estimate $H_k[\tilde u]$ at $x_0$.
Write
\beq
H_k[u]=\frac{\partial H_k}{\partial u_{nn}}u_{nn}+Q(\nabla u, u_{11},..., u_{ij}, ..., \widehat{u_{nn}}).
\eeq
The symbol `$\widehat{u_{nn}}$' means $Q$ is independent of $u_{nn}$.
Recall that $\frac{\partial H_k}{\partial u_{nn}}$ was given in Lemma 2.2.
By (2.20), (2.21) and (2.23), the quantity $Q\geq -C_1$ for $\tilde u$, where $C_1>0$ is a constant depending
on $u$, $v$ but independent of $\delta$.

Next we estimate $\frac{\partial H_k}{\partial u_{nn}}$ for $\tilde u$.
For simplicity, we denote by
$$D_{ij}=[1-(\frac{-t+\kappa}{\delta+\kappa})^\alpha s-\frac{t}{\delta}(1-s)]A_{ij}
       +[(\frac{-t+\kappa}{\delta+\kappa})^\alpha s+\frac{t}{\delta}(1-s)]B_{ij}.$$
and $\xi_i=u_i(p)=v_i(p)$, $1\leq i\leq n-1$.
By the formula (2.8) in Lemma 2.2 and (2.23), there exists $C_\delta>0$ such that
\beq
\frac{\partial H_k}{\partial u_{nn}}= \frac{1}{k!}\frac{\bar w^{k+1}}{w^{k+2}}\sum \left(\begin{array}{ccc}
     i_1 & \cdot\cdot\cdot & i_{k-1} \\
     j_1 & \cdot\cdot\cdot & j_{k-1}\end{array}\right)
     \bar a_{i_1j_1}\cdot\cdot\cdot \bar a_{i_{k-1}j_{k-1}}-C_\delta,
\eeq
for $\tilde u$ and $C_\delta\to 0$ as $\delta\to 0$, where
\beqs
\bar a_{ij}
  &=&\frac{D_{ij}}{\bar w}-\frac{\sum_{l=1}^{n-1}D_{il}\xi_l\xi_j}{\bar w^3}, \ 1\leq i, j\leq n-1,\\
\bar w &=& \sqrt{1+{\sum}_{i=1}^{n-1}\xi_i^2},\\
w &=& \sqrt{1+{\sum}_{i=1}^{n}\xi_i^2}.
\eeqs
By the $H_k$-subharmonicity of $u$, $v$, $\frac{\partial H_k}{\partial u_{nn}}\geq 0$ for $u$ and $v$. This implies the first term in (2.25) is nonnegative, i.e.,
$$\frac{\partial H_k}{\partial u_{nn}}\geq -C_\delta\to 0$$
for $\tilde u$.

Hence, by (2.21),
$$H_k[\tilde u] \geq -C_\delta \delta^{-1}-C_1.$$
at $x_0$,  which implies
$$\int_{\Omega\setminus \Omega_{\delta}}H_k[\tilde u]
\geq -C_\delta\frac{|\Omega\setminus\Omega_\delta|}{\delta}-C_1|\Omega\setminus\Omega_\delta|\to 0$$
as $\delta\to 0$. Therefore sending $\delta\to 0$ in (2.14), we have
$$\int_{\Omega}H_k[u]\geq \int_{\Omega}H_k[v].$$
\end{proof}


The following lemma is needed later to mollify a piecewise
smooth $H_k$-subharmonic function.
The proof is similar to that of Lemma 2.3.
Note that for a smooth function defined on one side of a smooth
hypersurface $\Gamma$,  if it is smooth up to $\Gamma$,
then one can extend it to the other side of $\Gamma$ by Taylor's expansion
in the normal bundle.

\begin{lem}
Let $v, v'$ be two smooth $H_k$-subharmonic functions in $\Om$.
Let $u=\max \{v, v'\}$ and denote $\Gamma=\{x\in\Om\ |\ v(x)=v'(x)\}$.
Assume $\Gamma$ is a smooth hypersurface, $Dv\ne Dv' $ on  $\Gamma$,
and $H_k[v], H_k[v']>0$ near $\Gamma$.
Then there is a sequence of smooth $H_k$-subharmonic functions $u_j$
which converges to $u$ and $u_j=u$ outside a small neighbourhood of $\Gamma$.
\end{lem}

\begin{proof}
First, we show that $u$ can be approximated by $C^{1, 1}$ smooth $H_k$-subharmonic functions.
Denote $\omega=\{x\in\Om \ |\  v<v'\}$. Then $\Gamma=\partial\omega$.
For any $p\in\partial\omega$, denote $\beta'=D_\gamma v'(p)$,
$\beta=D_\gamma v(p)$, where $\gamma$ is the unit normal of $\Gamma$.
Since $Dv\ne Dv' $ on  $\Gamma$, we have $\beta>\beta'$.
We choose a smooth function $\eta(p)$ on $\partial \omega$
such that $0<\eta<\beta-\beta'$ and $\eta=0$ on $\partial\Gamma$. 
Now for $p\in \partial \omega$ and small $\epsilon>0$, let
$$L(p, \epsilon)=\{p+t\gamma \ |\ -\epsilon\eta(p)\leq t\leq \epsilon\eta(p)\}, 
\text{\ and}\ \ \omega_\epsilon=\bigcup_{p\in\Gamma} L(p,\epsilon).$$
It is clear that any $x\in\omega_\epsilon$, there exists $p\in\partial\omega$
and $-\epsilon\eta\leq t\leq \epsilon\eta$ such that $x=p+t\gamma(p)$.
Note that by the smoothness of $\Gamma$, $x$ corresponds to a unique $(p,t)$. Then we define
\beq
\varphi(x)=\frac{\beta(p)-\beta'(p)}{4\epsilon\eta(p)}(t+\epsilon\eta(p))^2
\eeq
which is a smooth function in $\omega_\epsilon$, and
\beq
u^\epsilon(x)=
\begin{cases}
    u(x) + \varphi(x), & \ x\in{\omega_\epsilon}, -\epsilon\eta(p)\leq t\leq 0, \\[4pt]
    u(x) + \varphi(x)-(\beta(p)-\beta'(p))t, & \ x\in{\omega_\epsilon}, 0\leq t\leq \epsilon\eta(p),\\[4pt]
    u(x), & \ x\in\Omega\setminus\omega_\epsilon.
\end{cases}
\eeq
It is obvious that $u^\epsilon\in C^{1,1}$.
We show that $u^\epsilon$ is $H_k$-subharmonic in $\omega_\epsilon$ and converges to $u$
as $\epsilon\to 0$.
For any $x_0\in\omega_\epsilon$,
we have $x_0\in L (p, \epsilon)$ for some $p\in\Gamma$.
As above, we assume that $p=0$, and $\gamma(p)=(0,\cdots, 0, 1)$.
Denote $A_{ij}=v_{ij}(p)$, $B_{ij}= v'_{ij}(p)$.
Then near $0$, $\Gamma$ is given by
\beq
x_{n}=\rho(x')=\sum_{i, j=1}^{n-1}\frac{B_{ij}-A_{ij}}{2(\beta-\beta')}x_ix_j+O(|x'|^3).
\eeq
As in the proof of Lemma 2.3,  we have, at $x_0$,
\begin{eqnarray}
&&|u^\epsilon_i-u_i|\leq C\epsilon, \ |u^\epsilon_{ni}-u_{ni}|\leq C, \ 1\leq i\leq n-1,\\
&&|u^\epsilon_n-u_n|\leq C,\  u^\epsilon_{nn}-u_{nn}\geq \frac{1}{2\epsilon},\\
&&u^\epsilon_{ij}=(1-\frac{t+\epsilon\eta}{2\epsilon\eta})B_{ij}+\frac{t+\epsilon\eta}{2\epsilon\eta}A_{ij}+O(\epsilon), \ 1\leq i, j\leq n-1
\end{eqnarray}
for sufficiently small $\epsilon$.
To estimate $H_k[u^\epsilon](x_0)$, we again use formula (2.24)
\beq
H_k[u]=\frac{\partial H_k}{\partial u_{nn}}u_{nn}+Q(\nabla u, u_{11},..., u_{ij}, ..., \widehat{u_{nn}}).
\eeq
By (2.29)-(2.31), the quantity $Q\geq -C_1$,
where $C_1$ is a positive constant independent of $\epsilon$.
By the ellipticity assumption that $H_k[v], H_k[v']>0$ in the lemma,
there exists $C_2>0$, such that
$\frac{\partial H_k}{\partial u_{nn}}>C_2$ for $u$ and $v$.
With (2.29)-(2.31), we have $\frac{\partial H_k}{\partial u_{nn}}>\frac{C_2}{2}$ for $u^\epsilon$
provided $\epsilon$ is sufficiently small.
Then
$$H_k[u^\epsilon](x_0)\geq \frac{C_2}{2}\epsilon^{-1}-C_1$$
which implies $H_k[u^\epsilon](x_0)>0$.
Hence, we obtain a sequence of $C^{1, 1}$ smooth $H_k$-subharmonic functions
which converges to $u$.

Next, we construct the $C^{2, 1}$ approximation. But the above construction, it suffices to
consider $u=\max\{v,  v'\}$ in the lemma with further assumption that $u\in C^{1}(\Omega)$.
Then $\omega=\{x\in\Om \ |\  v<v'\}$ and $Dv=Dv'$ on $\Gamma=\partial \omega$. For any $p\in\partial\omega$, denote $\beta=D_{\gamma\gamma} v(p)$,
$\beta'=D_{\gamma\gamma} v'(p)$, where $\gamma$ is the unit normal of $\Gamma$. We have $\beta-\beta'>0$.
Choose a positive smooth function $\eta$  on $\partial \omega$
and $\eta=0$ on $\partial\Gamma$. 
We use the same notations $L(p, \epsilon)$ and $\omega_\epsilon$ as in $C^{1, 1}$ approximation.  
Let
\beq
u^\epsilon(x)=
\begin{cases}
    u(x) + \varphi(x), & \ x\in{\omega_\epsilon}, -\epsilon\eta(p)\leq t\leq 0, \\[4pt]
    u(x) + \psi(x), & \ x\in{\omega_\epsilon}, 0\leq t\leq \epsilon\eta(p),\\[4pt]
    u(x), & \ x\in\Omega\setminus\omega_\epsilon
\end{cases}
\eeq
where 
\beqs
\varphi(x)&=&\frac{\beta-\beta'}{24\epsilon\eta(p)}(t+\epsilon\eta(p))^3,\\
\psi(x)&=&-\frac{13(\beta-\beta')}{24\epsilon\eta(p)}(t-\epsilon\eta(p))^3-\frac{3(\beta-\beta')}{4(\epsilon\eta(p))^2}(t-\epsilon\eta(p))^4-\frac{\beta-\beta'}{4(\epsilon\eta(p))^3}(t-\epsilon\eta(p))^5.
\eeqs
One can check that $u^\epsilon\in C^{2,1}$. 
We show that $u^\epsilon$ is $H_k$-subharmonic in $\omega_\epsilon$ and converges to $u$
as $\epsilon\to 0$.  
For any $x_0\in\omega_\epsilon$,
we have $x_0\in L (p, \epsilon)$ for some $p\in\Gamma$.
By a transformation, we assume that $p=0$, and $\gamma(p)=(0,\cdots, 0, 1)$. 
By the assumption, we have 
\beqs
&&Dv(p)=Dv'(p)=Du(p),\\ 
&&v_{ij}(p)=v'_{ij}(p)=u_{ij}(p),  i, j=1, ..., n-1,\\ 
&&v_{in}(p)=v'_{in}(p)=u_{in}(p), i=1, ..., n-1.
\eeqs 
The main point is that under the construction (2.33), $Du$, $\{u_{ij}\}_{i, j=1}^{n-1}$,
$\{u_{in}\}_{i, j=1}^{n-1}$  change very little, while $u_{nn}$ will always 
be no less than $\min\{v_{nn}, v'_{nn}\}$. First, by the above facts, 
$\Gamma$ is given by
\beq
x_{n}=\rho(x')=\sum_{i, j=1}^{n-1}A_{ij}x_ix_j+O(|x'|^3)
\eeq
near $0$ for some $A_{ij}$, $i, j=1, ..., n-1$.
Then by similar arguments and computation as before, we have
\begin{eqnarray}
&&|u^\epsilon_i(x_0)-v'_i(p)|\leq C\epsilon, \ , \ 1\leq i\leq n,\\[3pt]
&&|u^\epsilon_{ni}(x_0)- v'_{ni}(p)|\leq C\epsilon,\  |u^\epsilon_{ij}(x_0)-v'_{ij}(p)|\leq C\epsilon, \ 1\leq i, j\leq n-1,
\end{eqnarray}
for sufficiently small $\epsilon$.
Furthermore, since $\varphi_{nn}>0$ and
\beqs
\psi_{nn}&=&-\frac{13(\beta-\beta')}{4\epsilon\eta(p)}(t-\epsilon\eta(p))-\frac{9(\beta-\beta')}{(\epsilon\eta(p))^2}(t-\epsilon\eta(p))^2-\frac{5(\beta-\beta')}{(\epsilon\eta(p))^3}(t-\epsilon\eta(p))^3\\
&>&-(\beta-\beta'),
\eeqs
we have
\beq
u^\epsilon_{nn}(x_0)\geq  v'_{nn}(p)-C\epsilon
\eeq
for sufficiently small $\epsilon$.
Using formula (2.24) with (2.35)-(2.37) and the ellipticity assumption that $H_k[v], H_k[v']>0$, we have
$$H_k[u^\epsilon](x_0)\geq H_k[v'](p)-C\epsilon>0$$
as $\epsilon \to 0$.

By choosing the function $\varphi$ more carefully, the sequence can be made
$C^{l, 1}$ for any $l\geq 3$. Actually, the $C^{2, 1}$ approximation is enough for the purpose of this paper.
\end{proof}

We point out that for $l \geq 2$, the construction of the $C^{l, 1}$ approximation by small modification
is complicated. If we do not request $u_j=u$ outside a small neighbourhood of $\Gamma$, then there are
other ways to get the $C^{\infty}$ approximation from the $C^{1, 1}$ approximation. One way is to mollify the $C^{1, 1}$ function $u^\eps$ constructed in the proof of Lemma 2.4 directly by convolution. The 
The other way is to consider the initial-boundary value problem to an associated parabolic equation
and obtain a smooth solution $u(x, t)$ with initial condition $u^\eps$.
Then $u(\cdot, t)$, as $t\to 0$, gives another smooth approximation.

\section{Perron lifting}


To prove Theorem 1.1,
we assume that there exists two sequences of
bounded $H_k$-subharmonic functions $\{u_j\}$, $\{v_j\}$
in $\mathcal{SH}_k(\Omega)$
which converge to an $H_k$-subharmonic function $u$ a.e. in $\Omega$.
Let $B_{r}(x_0)$ and $B_{r+t}(x_0)$ be two balls in $\Om$.
The purpose of this section is to modify $u_j$ and $v_j$ in the annulus
$B_{r+t}-B_r$ such that they are locally uniformly Lipschitz continuous and
$|u_j-v_j|\to 0$ locally uniformly in the annulus.

For this purpose we use the Perron lifting for $k$-curvature equations,
following the treatment in \cite{DTW} for the mean curvature equation.
See also \cite{HKM} for quasilinear elliptic equations.
As the argument is very similar, we will sketch the proof only.

However let us point out a difference, that is for the mean curvature equation,
by the interior regularity one obtains a sequence of piecewise smooth functions
in the annulus. For the $k$-curvature for $1<k<n$,
by the interior gradient estimate \cite {K, W1},
we only obtain a sequence of piecewise Lipchitz continuous functions
in the annulus. In the next section we will show how
to obtain a smooth approximation in the annulus.

Let $u$ be a $H_k$-subharmonic function in $\Omega$ and
$\omega\Subset\Omega$ be a subdomain of
$\Omega$. The {\it Perron lifting} of $u$, $u^\omega$,
is the upper semicontinuous regularization \cite{HKM} of
$$\tilde u=\{v \ |\  \ \text{$v$ is $H_k$-subharmonic in $\Omega$
and $v\leq u$ in $\Omega\setminus\omega$}\},$$
i.e.,
$$u^\omega(x)=\lim_{t\to 0}\sup_{B_t(x)}\tilde u.$$
It is clear that $u^\omega\geq u$ in $\Omega$ and $u^\omega=u$ in $\Omega\setminus\bar\omega$ but
$u^\omega=u$ on $\partial\omega$ may not be true.

In order to study the properties of Perron lifting,
we first recall the existence of solutions to the Dirichlet problem for $k$-curvature equations \cite{I2, T1, T2}.
For any $C^2$ domain, we denote by $H_{k}[\partial\Omega]$
the $k$-curvature of the boundary $\partial\Omega$.
Let us quote the following two lemmas (Theorem 4 and 5, \cite{T2}).

\begin{lem}
Assume $\partial\Omega\in{C^{2}}$, $\varphi\in C^{0}(\partial\Omega)$,
$f^{\frac{1}{k}}\in{C^{1,1}(\overline\Omega)}$, $f>0$ in $\Omega$.
Suppose
\begin{eqnarray}
f(x)& \leq & H_{k}[\partial\Omega],\ \  x\in\partial\Omega,\\
\int_E f &\leq & \frac{1-\lambda}{k}H_{k-1}[\partial E]
\end{eqnarray}
for some $\lambda>0$ and
all subdomains $E\subset\Omega$ with $(k-1)$-convex boundary $\partial E\in C^2$.
Then there exists a unique, $H_k$-subharmonic,
viscosity solution $u\in C^0(\overline \Omega)$, which is locally
uniformly Lipschitz continuous in $\Omega$
to the Dirichlet problem
\beq
\begin{cases}
   H_k[u]=f(x)   & \text{\ \  in $\Omega$}, \\
    u=\varphi  & \text{\ \ on $\partial\Omega$}.
\end{cases}
\eeq
If $\varphi\in C^{1,1}(\overline\Omega)$, then $u\in C^{0,1}(\overline\Omega)$.
\end{lem}

\begin{lem}
Suppose that $\partial\Omega\in{C^{3,1}}$, $\varphi\in C^{3,1}(\overline\Omega)$,
$f\in{C^{1,1}(\overline\Omega)}$, $f>0$ in $\Omega$.
Suppose (3.1), (3.2) hold. Then there exists a unique, $H_k$-subharmonic,
classical solution to (3.3).
\end{lem}

\begin{rem}
It is shown in \cite{T2} that (3.2) is a necessary condition for the solvability of
 $k$-curvature equations.
\end{rem}

By choosing a sequence of $f_i\to 0$ and using approximation, we have

\begin{cor}
Assume $\partial\Omega\in{C^{2}}$, $H_{k}[\partial\Omega]>0$
and $\varphi\in C^{1,1}(\overline\Omega)$.
Then there exists a unique, $H_k$-subharmonic,
viscosity solution $u\in C^{0,1}(\overline \Omega)$ to the Dirichlet problem
\beq
\begin{cases}
   H_k[u]=0  & \text{\ \  in $\Omega$}, \\
    u=\varphi  & \text{\ \ on $\partial\Omega$}.
\end{cases}
\eeq
\end{cor}

We have the following existence and Lipschitz continuity
for the Perron lifting.

\begin{lem}
Let $u\in\mathcal{SH}_k(\Omega)$.
Then for any open set $\omega\Subset\Omega$,
the Perron lifting $u^\omega$ is $H_k$-subharmonic function in $\Omega$.
If we further assume $u\in L_{loc}^\infty(\Omega)$, then $u^\omega$ is
locally Lipchitz continuous in $\omega$ for any open set $\omega\Subset\Omega$.
\end{lem}

\begin{proof}
The property that $u^\omega$ is $H_k$-subharmonic function in $\Omega$ follows by definition.
It suffices to show that $u^\omega$ locally Lipchitz continuous in $\omega$ when $u\in L_{loc}^\infty(\Omega)$.
We outline the proof here, as it is similar to that in \cite {DTW, HKM}. The idea is as follows.
Since $u$ is upper semicontinuous, there exists a sequence $\{\varphi_j\}$ of smooth functions in $\Omega$
such that $\varphi_j\searrow u$. Assume $B\Subset \omega$ is a ball. By Corollary 3.4,
there is a solution $u_j\in C^{0,1}(\overline B)$  to
\beq
\begin{cases}
   H_k[u]=0  & \text{\ \  in $B$}, \\
    u=\varphi_j  & \text{\ \ on $\partial B$}.
\end{cases}
\eeq
 Furthermore, since $u$ is bounded,
 by the interior gradient estimate \cite {K, W1},
 the decreasing sequence $u_j$ is locally uniformly Lipchitz.
 Hence $u_j$ converges locally, uniformly
to a locally Lipchitz function $\hat u$ .
It is obvious that $u_j\ge u$ and hence $\hat u\geq u^\omega\geq u$.
By a barrier construction, we can show that $\hat u\leq u$ on $\partial B$, in the sense that
for any given $x_0\in \partial B$, $\lim_{x\to x_0} \hat u(x)\leq u(x_0)$.
Now extend $\hat u$ to $\Omega$ so that $\hat u=u$ in $\Omega\setminus B$.
$\hat u$ is $H_k$-subharmonic in $\Omega$,
which implies that $\hat u=u^\omega$ by the definition of $u^\omega$.
\end{proof}

We also need the following convergence of the Perron lifting.

\begin{lem}
Let $u_j$ be a sequence of uniformly bounded $H_k$-subharmonic
functions which converges to $u\in\mathcal{SH}_k(\Omega)$ a.e.
Let $B_{r_0}(x_0)\subset \Om$. Then for a.e. $r\in (0, r_0)$,
we have $u_j^{B_r}\to u^{B_r}$ a.e. in $\Om$, as $j\to\infty$.
\end{lem}

The proof is by the monotonicity of $u^{B_r}$ for $r\in (0, r_0)$,
i.e, $u^{B_r}$ is increasing in $r$ and
$$\lim_{r\to \delta^-}u^{B_r}\leq u^{B_\delta}\leq \lim_{r\to \delta^+}u^{B_r}.$$
It is similar to that in \cite {DTW}. We refer the reader to \cite {DTW}
for details.


\section{Approximation}


In this section, we prove that every $H_k$-subharmonic function can be approximated
by a sequence of smooth $H_k$-subharmonic functions.
The approximation for the case $k=1$ was obtained in \cite{DTW}
by the Perron lifting in small balls and using the mollification
as in the proof of Lemma 2.4.
However, when $k>1$,  the Perron lifting of an $H_k$-subharmonic function
may fail to be smooth.
In this paper we introduce a new technique to prove the approximation,
by considering an obstacle problem of the equation.
The regularity of the solution to the obstacle problem is proved in the Appendix 1.

For  $u\in C^0(\overline\Omega)$, define the $H_k$-subharmonic envelope
\beq
\tilde u=\sup\{v\in \mathcal {SH}_k(\Omega)\cap C^0(\Omega)\ | \ \text{$v\leq u$ in $\Omega$}\}.
\eeq
It is the greatest subsolution to an obstacle problem for the $k$-curvature equation.

\begin{theo}
Let $u\in\mathcal{SH}_k(\Omega)$. Then for any ball $B_R\Subset\Omega$,
there is a sequence of smooth $H_k$-subharmonic functions
$u_j\in C^\infty(B_R)$
converging to $u$ in $B_R$.
\end{theo}

\begin{proof}
We may assume that $u$ is continuous.
Indeed, by the property of Perron lifting in last section,
we may use the techniques as Theorem 5.1 in \cite{DTW}
to construct a sequence of piecewise Lipchitz, $H_k$-subharmonic functions
to approximate it.

To obtain the smooth approximation, we first choose
a sequence $\{\varphi_j\}$ of smooth functions in $\Omega$
such that $\varphi_j\searrow u$ and $\varphi_j>u$.
Let $\tilde \varphi_j$ be the $H_k$-subharmonic envelope of $\varphi_j$ in $B_R$,
as defined above.
Then $\varphi_j\geq \tilde \varphi_j\geq u$.
Hence, $\tilde\varphi_j$ converges to $u$ as $j\to \infty$.

But the function $\tilde\varphi_j$ is not smooth.
To obtain a smooth approximation,
we choose a sequence $\delta_j>0$, converging to $0$ as $j\to\infty$,
and assume that $\varphi_j\geq u+\delta_j$. Now we consider the obstacle problem
\beq
\varphi_j^{\delta_j}=\sup_{v\in S_{\varphi_j, \delta_j}} v .
\eeq
where
$$S_{\varphi_j, \delta_j}=\{v\in \mathcal {SH}_k(\Omega)\cap C^0(B_R)\ |\
\ \text{$v\leq \varphi_j$ in $B_R$ and $v\leq \varphi_j-\delta_j$ on $\partial B_R$}\}.$$
We will show that (Theorem 6.1 in the Appendix 1)
$\varphi_j^{\delta_j}$ is a Lipchitz $H_k$-subharmonic smooth function.
Furthermore, for any  $j$,
there exists a sequence of smooth $H_k$-subharmonic  functions
$u_j^\epsilon$, which converges to $\tilde \varphi_j^{\delta_j}$ as $\epsilon \to 0$.
Since $u\in S_{\varphi_j, \delta_j}$, we have
$\tilde \varphi_j\geq\tilde \varphi_j^{\delta_j}\geq u$.
Hence, $u_j^\epsilon$ converges to $u$ as $\epsilon \to 0$, $j\to \infty$.
We obtain the smooth approximation.
\end{proof}

\begin{rem}
Note that the function $u$ in Theorem 4.1
satisfies $H_k[u]\ge 0$ in the viscosity sense.
Namely we allow that $u$ is degenerate in the sense that $H_k[u]=0$ at some points.
But by our proof of Theorem 6.1,
the smooth function $u_j$ obtained in Theorem 4.1 satisfies $H_k[u_j]>0$ in $\Om$.
This important property enables us to assume that the sequence $\{u_j\}$
in Theorem 1.1 are smooth and  $H_k[u_j]>0$. In other words, by Theorem 4.1,
it suffices to prove Theorem 1.1 (ii) for any sequence of {\it smooth}, bounded,
strictly $H_k$-subharmonic functions $\{u_j\}$ satisfying $H_k[u_j]>0$.
\end{rem}

\section{Weak continuity}


In this section, we prove Theorem 1.1.
First we prove

\begin{lem}
Let $u_j$ be a sequence of $H_1$-subharmonic functions
which converges to an $H_1$-subharmonic function $u$ a.e..
Then for any $\eps>0, \delta>0$, and any subset
$\Om'\subset\subset \Om$,
there exists $J>1$ such that when $j>J$, we have
$$u_j(x)\le u_\delta(x)+\eps\ \ \ \forall\ x\in\Om', $$
where
$$u_\delta(x)=\sup\{u(y)\ |\ |y-x|<\delta\}. $$
\end{lem}

\begin{proof} For any given $x_0\in\Om'$,
by subtracting a constant we assume that  $u_\delta(x_0)=0$.
Replacing $u_j$ by $\max(u_j, 0)$ we may assume that
$u_j\ge 0$ in $B_\delta(x_0)$.
By the Harnack inequality (see Corollary 2.1 in \cite {DTW}), we have
$$u_j(x_0)\le C\delta^{-n/p} \Big(\int_{B_\delta} u_j^p \Big)^{1/p}. $$
But since $u_j$ are uniformly bounded and $u_j\to u$ a.e.,
the right hand side converges to $0$.
Hence $u_j(x_0)\le \eps$ when $j$ is sufficiently large.
\end{proof}

To prove Theorem 1.1, by Remark 4.2 it suffices to
consider a sequence $\{u_j\}$ which are smooth and satisfy $H_k[u_j]>0$.
When no confusion arises,
we also use the density $H_k[u]$ to denote the measure
$\mu_k[u]=H_k[u]\,dx$ for $u\in\mathcal{SH}_k\cap C^2$.

To prove Theorem 1.1, it also suffices to prove part (ii) of it.
Therefore we need only to prove

\begin{lem}
Let $u_j\in C^2(\Omega)$ be a sequence
of  uniformly bounded, $H_k$-subharmonic functions.
Suppose $u_j$ converges to a strictly $H_k$-subharmonic function
$u\in\mathcal{SH}_k(\Omega)$ a.e..
Then $H_k[u_j]$ converges to a measure $\mu$ weakly.
\end{lem}

\begin{proof}
For any ball $B_r(x_0)\subset\subset \Om$,
let $\phi(x)=\lambda (|x-x_0|-r)$, which is the representative function of a convex cone. We choose $\lambda$
sufficiently large such that $\phi<u$ in $B_{r-\delta}(x_0)$ and
$\phi>u$ on $\p B_{r+\delta}(x_0)$,
where $\delta=\lambda^{-1}\sup_j\sup_\Om |u_j|$.
By Lemma 2.3,  we have
$$\int_{B_{r-\delta}(x_0)} H_k[u_j]\le
 \int_{B_{r+\delta}(x_0)} H_k[\max(u_j, \phi)]=\int_{B_{r+\delta}(x_0)} H_k[\phi],$$
which is uniformly bounded. This means that $\int_{\Omega'}H_k[u_j]$ are uniformly bounded
for $\Omega'\Subset\Omega$.
Hence, according to \cite{A} there is a subsequence of $H_k[u_j]$ which converges to
a measure $\mu$ weakly.

To prove that the convergence is independent of the subsequences,
we consider two sequences
$\{u_j\}$, $\{v_j\}$ in $\mathcal{SH}_k(\Omega)\cap C^2(\Omega)$.
Suppose both of them converge to $u$ a.e. in $\Omega$ and
\beq
H_k[u_j]\longrightarrow \mu, \ \ H_k[v_j]\longrightarrow \nu.
\eeq
Then to prove $\mu=\nu$,
it suffices to prove that for any ball $B_r=B_r(x_0)\subset\subset\Omega$
and any $t>0$,
\beq
\mu(B_{r})\leq \nu(B_{r+t}),\ \ \nu(B_{r})\leq \mu(B_{r+t}).
\eeq

Our strategy of proving (5.2) is to use the Perron lifting in the annuli
$N_{t/8}(\p B_{r+t/4})$ and $N_{t/8}(\p B_{r+3t/4})$
and use the monotonicity formula (2.9),
where $N_\delta$ denotes the $\delta$-neighbourhood.
As can be seen in Lemma 3.1 and 3.2, the $H_k$-Perron lifting directly in an annuli domain
can not guarantee the regularity. We will use the Perron lifting in a modified way
by liftings in a finite sequence of finite covering balls. We also use Lemma 2.4 to get the smoothness near
the boundaries of these balls, so that the Perron lifting of $u_j$ and $v_j$ are smooth.
Details is as follows.

First we choose finitely many balls $\{\hat B_{r_\alpha}\}_{\alpha=1}^m$
of radius $r_\alpha\approx \frac {t}{8}$ and
center $x_\alpha$ on $\partial B_{r+{t}/{4}}$, such that
$$
 N_{t/16}(\p B_{r+t/4}) \subset {\bigcup}_{\alpha=1}^{m} \hat B_{r_\alpha}
         \subset N_{t/8}(\p B_{r+t/4}).
$$
Denote $\hat u_{j,0}=u_j$.
Let $\hat u'_{j, 1}$ be the solution to
\beq
\begin{cases}
    H_k[u] =\sigma_{j,1} &  \ \ \text{in}\  \hat B_{r_1}, \\
     u=\hat u_{j,0} & \ \ \text{in}\  \Omega\setminus \hat B_{r_1},
\end{cases}
\eeq
where $\sigma_{j,1}\leq \inf H_k[\hat u_{j,0}]$ is a constant sufficiently small
but positive,  such that the Dirichlet problem (5.3) is solvable.
Then $\hat u'_{j, 1}$ is smooth in $\hat B_{r_1}$ and
in $B_{r+t}\setminus \hat B_{r_1}$, up to the boundary.
By Lemma 2.4 we may modify
$\hat u'_{j,1}$ slightly near $\p B_{r_1}$ to get a smooth $H_k$-subharmonic function
$\hat u_{j,1}$ in $\Om$.

We define $\hat u_{j, \alpha}$ inductively for $\alpha=2, 3, \cdots, m$,
such that $\hat u_{j, \alpha}\in \mathcal{SH}_k(\Om)$
is a modification of $\hat u'_{j,\alpha}$ near $\p \hat B_{r_\alpha}$,
where $\hat u'_{j,\alpha}$ is the solution to
\beq
\begin{cases}
    H_k[u] =\sigma_{j,\alpha} &  \ \ \text{in}\   \hat B_{r_\alpha} , \\
     u=\hat u_{j,\alpha-1} & \ \ \text{in}\   \Omega\setminus \hat B_{r_\alpha}.
\end{cases}
\eeq
for some sufficiently small positive constant
$\sigma_{j,\alpha}<\inf H_k[\hat u_{j,\alpha-1}]$.

Denote $\hat u_j=\hat u_{j,m}$. Similarly, we obtain $\hat v_j$.
Both $\hat u_j$ and $\hat v_j$ are smooth and $H_k$-subharmonic.
But note that the function $u$ is not smooth and $H_k[u]$ may vanish at some points.
So we simply let $\hat u_{0}=u$ and
for $\alpha=1,2,\cdots, m$, let $\hat u_{\alpha}$ be the solution to
\beq
\begin{cases}
    H_k[u] =0 &  \  \ \text{in}\  \hat B_{r_\alpha} , \\
     u=\hat u_{\alpha-1} & \ \ \text{in}\  \Omega\setminus \hat B_{r_\alpha}.
\end{cases}
\eeq
and denote $\hat u=\hat u_{m}$.

In conclusion, the almost everywhere convergent sequences
$\{u_j\}$, $\{v_j\}$ become uniformly convergent sequences $\{\hat u_j\}$, $\{\hat v_j\}$
near $\p B_{r+t/4}$ after the above process.
In fact, by Lemma 3.6, we may choose the radii $r_\alpha\approx\frac t8$ and choose
$\sigma_{j, \alpha}$ sufficiently small such that
$\hat u_j\to \hat u$ and $\hat v_j\to \hat u$
uniformly in $N_{t/16}(\p B_{r+t/4})$.
Moreover, by the interior gradient estimate for the $k$-curvature equation \cite{K, W1},
$\hat u_j$, $\hat v_j$ and $\hat u$ are locally uniformly Lipchitz continuous
in $\bigcup_{\alpha=1}^{m}\hat B_{r_\alpha}$.
Hence by subtracting a small constant we may assume that
\beq
\hat u_j<\hat v_j \ \ \ \ \text{on}\ \ \p B_{r+t/4}.
\eeq

Next we apply the above modified Perron lifting to the functions
$\hat u_j$ and $\hat u$ near $\p B_{r+3t/4}$.
That is we choose finitely many balls $\{\tilde B_{r_\beta}\}_{\beta=1}^{m'}$
of radius $r_\beta\approx \frac {t}{8}$ and
center $x_\beta$ on $\partial B_{r+{3t}/{4}}$, such that
$$N_{t/16}(\p B_{r+3t/4})
\subset {\bigcup}_{\beta=1}^{m'} \tilde B_{r_\beta} \subset N_{t/8}(\p B_{r+3t/4}). $$
Let $\tilde u_{j, 0}=\hat u_j$. For $\beta=1, 2, 3, \cdots, m'$, let
$\tilde u'_{j,\beta}$ be the solution to
\beqs
\begin{cases}
    H_k[u] =\sigma_{j,\beta} &  \ \ \text{in}\   \tilde B_{r_\beta} , \\
     u=\tilde u_{j,\beta-1} & \ \ \text{in}\  \Omega\setminus \tilde B_{r_\beta};
\end{cases}
\eeqs
and let $\tilde u_{j, \beta}\in \mathcal{SH}_k(\Om)$
be a modification of $\tilde u'_{j,\beta}$ near $\p B_{r_\beta}$
such that it is smooth and $H_k$-subharmonic.
Denote $\tilde u_j=\tilde u_{j, m'}$.
Similarly let $\tilde u_{0}=\hat u$ and
for $\beta=1,2,\cdots, m'$, let $\tilde u_\beta$ be the solution to
\beqs
\begin{cases}
    H_k[u] =0 &  \  \hat B_{r_\beta} , \\
     u=\tilde u_{\beta-1} & \ \Omega\setminus \hat B_{r_\beta}.
\end{cases}
\eeqs
and denote $\tilde u=\tilde u_{m'}$.

In the above we obtained by the Perron lifting the new sequences
$\tilde u_j, \hat u_j, \hat v_j$, and the modified function $\hat u, \tilde u$
with the properties
$${\begin{aligned}
 &\tilde u_j=\hat u_j=u_j,\\
 &\tilde u =\hat u=u,\\
 &\hat v_j=v_j
 \end{aligned} }
$$
near $\p B_r\cup\p B_{r+t/2}\cup \p B_{r+t}$.
By assumption, $u$ is strictly $H_k$-subharmonic. Hence there exists $\eps>0$ such that
$\tilde u> u+\eps$ near $\p B_{r+3t/4}$. By Lemma 3.6,
$\tilde u_j\to \tilde u$ near $\p B_{r+3t/4}$.
By Lemma 5.1, we then have
\beq
\tilde u_j>\hat v_j=v_j\ \ \ \text{near} \ \ \p B_{r+3t/4}.
\eeq

By (5.6) and (5.7), there exist set $G_j$ with
\beq
B_{r+t/4}\subset G_j\subset B_{r+3t/4}
\eeq
such that $$\tilde u_j(x)=\hat v_j(x)$$ for $x\in\partial G_j$ and
$$\tilde u_j(x)< \hat v_j(x)$$
for $x\in G_j$ near $\p G_j$. By Sard's lemma,
we may also assume by subtracting a small constant to $\tilde u_j$
that $D\tilde u_j\ne D\hat v_j$ on $\p G_j$\footnote{Consider the set $G_\delta=\{\tilde u_j<\hat v_j+\delta\}$. By Sard's lemma, the set is smooth for almost all $\delta>0$. We can choose a smooth $G_\delta$ for some small $\delta$ and hence $D\tilde u_j\ne D\hat v_j$ on $\p G_\delta$.}.
Therefore by the monotonicity formula (2.9), we have
$${ \begin{aligned}
 \int_{B_r} H_k[v_j]
   & =  \int_{B_r}  H_k[\hat v_j]
 \le \int_{G_j}  H_k[\hat v_j]\\
 &  \le \int_{G_j}  H_k[\tilde u_j]
 \le \int_{B_{r+t}}  H_k[\tilde u_j]
 = \int_{B_{r+t}}  H_k[u_j].
 \end{aligned} }$$
Taking limit we obtain
$\nu(B_{r})\leq \mu(B_{r+t})$.
By exchanging $u$ and $v$ we also have
$\mu(B_{r})\leq \nu(B_{r+t})$.
\end{proof}

\begin{rem}
We point out that even if the sequence $\{u_j\}$
are locally uniformly Lipschitz continuous,
we cannot remove the strict $H_k$-subharmonicity condition,
because we cannot modify the sequence $u_j$ and $v_j$
to obtain (5.6) and (5.7) simultaneously.
\end{rem}

\begin{rem}
The inequalities (5.2) is built upon the monotonicity formula (2.8)
and we need the modified functions of $u_j$ and $v_j$
satisfies (5.6) and (5.7).
For the $k$-Hessian equation or the $p$-Laplace equations,
to obtain an inequality like (5.7),
one may simply add the function
$w_\eps:=\eps[\max(0, |x|-r-\frac t4)]^2$ to $u_j$.
However for the $k$-curvature equation,
$u_j+w_\eps$ is not $H_k$-subharmonic in general.

One may solve the  Dirichlet problem
\beq
\begin{cases}
    H_k[u] =0 &  \  \ \text{in}\  B_{r+3t/4}\setminus B_{r+t/4} , \\
     u=\tilde u_{j,\beta} +\eps & \ \ \text{on}\  \p B_{r+3t/4},\\
     u=\tilde u_{j,\beta}   & \ \ \text{on}\  \p B_{r+t/4}.
\end{cases}
\eeq
and send $\eps\to 0$. However we cannot prove the convergence of
the gradient of the solution on $\p B_{r+t/4}$.
For these reasons we have to assume that $u$ is strictly $H_k$-subharmonic.
\end{rem}


\section{Appendix 1: An obstacle problem}

Let $\Omega$ be a smooth, bounded domain in $\mathbb R^n$
with $H_k[\partial\Omega]>0$ and let
$\varphi\in C^{\infty}(\overline\Omega)$. Suppose $\delta>0$ is a small positive constant.
In this appendix, we consider greatest viscosity subsolution to the obstacle problem
\beq\begin{cases}
H_k[u]\geq  0  & \text{\ in $\Omega$},\\
u\leq \varphi &\text{\ in $\Omega$},\\
u\leq \varphi-\delta &\text{\ on $\partial\Omega$},
\end{cases}\eeq
i.e.,
\beq
\tilde \varphi^\delta=\sup_{u\in S_{\varphi,\delta}} u,
\eeq
where $$S_{\varphi,\delta}=\{u\in C(\Omega) \ |\  H_k[u]\geq  0,\ u\leq \varphi, \ u|_{\partial\Omega}\leq \varphi-\delta \}.$$
It is clear that $S_{\varphi,\delta}$ is nonempty, so $\tilde \varphi^\delta$ is well defined.

\begin{theo}
The function $\tilde \varphi^\delta$ is Lipchitz continuous on $\overline\Omega$, and satisfies
\beq\begin{cases}
H_k[\tilde \varphi^\delta]\geq  0  & \text{\ in $\Omega$},\\
\tilde \varphi^\delta\leq \varphi &\text{\ in $\Omega$},\\
\tilde \varphi^\delta= \varphi-\delta &\text{\ on $\partial\Omega$}.
\end{cases}\eeq

Furthermore, $\tilde \varphi^\delta$ can be approximated by a sequence of smooth $H_k$-subharmonic functions
$u^\epsilon$, which are solutions to the following perturbation problem:
\begin{equation}
\begin{cases}
\{H_k[u]\}^{\frac{1}{k}}=-\beta_\epsilon(\varphi-u)  & \text{\ \  in $\Omega$}, \\[5pt]
    u=\varphi-\delta  & \text{\ \ on $\partial \Omega$,}
\end{cases}
\end{equation}
where $\beta_\epsilon(t)$ is a family of smooth functions ($\epsilon>0$)
such that
\begin{eqnarray*}
&&\beta_\epsilon(0)=-1,\ \beta_\epsilon(t)< 0,\\
&&\beta_\epsilon(t)\to -\infty,\ t<0,\ \epsilon\to 0,\\
&&\beta_\epsilon(t)\to 0,\ t>0,\ \epsilon\to 0,\\
&&\beta'_\epsilon(t)\geq 0,\ \beta''_\epsilon(t)\leq 0.
\end{eqnarray*}
\end{theo}

An example of the function $\beta_\eps$
satisfying the above conditions is $\beta_\eps(t)=-e^{-\epsilon^{-1}t}$.
Theorem 6.1 is inspired by \cite{L},
which deals with an obstacle problem for Monge-Amp\`ere equation.

First let us recall the a priori estimates for the curvature equation \cite{I2}
\begin{equation}
\begin{cases}
   H_k[u]=f(x, u)  & \text{\ \  in $\Omega$}, \\[5pt]
    u=\varphi  & \text{\ \ on $\partial \Omega$}.
\end{cases}
\end{equation}

\begin{theo}{\cite{I2}}
Let $\Omega$ be a bounded domain in $R^n$ with $C^4$ boundary
and $\varphi\in C^4(\partial\Omega)$.
Assume $H_k[\partial\Omega]>0$,
$f\in C^2(\overline\Omega\times \mathbb R)$ and $f$ satisfies
\beq
f>0, \ \ \ \ \ \frac{\partial f}{\partial u}\geq 0
\eeq
and
\beq
f <  H_{k}[\partial\Omega]\ \ \ \text{on} \ \pom  .
\footnote{The condition was  given by $f < \frac{n-k}{n} H_{k}[\partial\Omega]$ in \cite{I2}
since the $k$-curvature was defined to be
$\left(\begin{array}{c}n \\ k\end{array}\right)^{-1}\sigma_k(\kappa)$ there, where
$\kappa=(\kappa_1,\cdot\cdot\cdot, \kappa_n)$ is the vector of
principal curvatures.}
\eeq
Let $u\in C^2(\overline\Omega)$ be a solution to (6.5).
Denote $M=\sup_\Omega |u|$, $\nu=\inf f(x, u)>0$.
Then there exists $C>0$ depending on
$M$, $\nu$, $\|\varphi\|_{C^2(\partial\Omega)}$,
$\|\partial\Omega\|_{C^4}$, $\|f\|_{C^2}$ and $H_{k}[\partial\Omega]-f|_{\partial\Omega}$, such that
\beq
\|u\|_{C^2(\overline\Omega)}\leq C.
\eeq
\end{theo}

Applying Theorem 6.2  to equation (6.5) we have the following existence result.

\begin{theo}
For any small $\epsilon>0$, there exists a smooth $H_k$-subharmonic solution $u^{\epsilon}$ to Dirichlet problem (6.4).
\end{theo}

\begin{proof}
Let $f(x, u)=-\beta_\epsilon(\varphi-u)$.
It is clear that
$$\frac{\partial f}{\partial u}=\beta_\epsilon'(\varphi-u)\geq 0.$$
On the boundary $\partial\Omega$,
$$f(x, u)=-\beta_\epsilon(\delta) <H_{k}[\partial\Omega]$$
provided $\epsilon$ is small enough.
Thus conditions (6.6), (6.7) are satisfied.
Hence by Theorem 6.2 and the regularity theory of Evans and Krylov \cite{E, Kr},
we obtain a priori estimates in $C^{3,\alpha}(\bom)$.
By the continuity method, we obtain the existence of
solutions to (6.4).
\end{proof}

\vskip10pt

\noindent{\it Proof of Theorem 6.1}.
Let $u^\epsilon$ be the smooth solution to (6.4).
In the following we show that the $C^0$ and $C^1$ bounds are
independent of $\eps>0$ small.
First we show
\beq
\sup_\Omega |u^\epsilon|\leq C
\eeq
for some constant $C>0$ independent of $\eps>0$.

Indeed, by the $H_k$-subharmonicity,
$\sup_\Omega u^\epsilon=\sup_{\partial\Omega}\varphi$. Hence,
we only need to prove $u^\eps$ is bounded from below.
By Lemma 3.1, there exists  $\sigma_0>0$ small enough such that the
Dirichlet problem
\begin{equation}
\begin{cases}
   H_k[u]=\sigma_0^k  & \text{\ \  in $\Omega$}, \\[5pt]
    u=\varphi-\delta  & \text{\ \ on $\partial\Omega$}
\end{cases}
\end{equation}
is solvable. Let $u_0$ be the unique solution to (6.10).
Let $K>0$ be a positive constant and denote
$$E=\{x \ |\   u^\epsilon(x)\leq \varphi(x)-K\}.$$
Then $E\subset \Omega$.
By the definition of $\beta_\epsilon$,
$$f=-\beta_\epsilon(\varphi-u^\epsilon)\leq -\beta_\epsilon(K)\leq\sigma_0$$
when $K$ is chosen large enough.
Hence by the comparison principle,
\beq
u^\epsilon\geq u_0-C_K
\eeq
for some constant $C_K$ depending on $\varphi$ and $K$.
Hence (6.9) holds for some constant $C>0$ independent of $\epsilon$.

\vskip 10pt

Next we show that
\beq
\sup_\Omega |\nabla u^\epsilon|\leq C
\eeq
for some constant $C>0$ independent of $\epsilon$.
The gradient estimate can be found in \cite {CNS}.
We include the proof, only to show that the upper bound is independent of $\eps$.

Let us first show that $f(x, u^{\epsilon})$ is uniformly bounded from above.
Indeed, assume that the minimum of $\beta_\epsilon(\varphi-u^{\epsilon})$ is attained at a point $x_0$.
If $x_0\in\partial \Omega$, by the boundary condition,
$\beta_\epsilon(\varphi-u^{\epsilon})(x_0)=\beta_\epsilon(0)=-1$.
If $x_0\in \Omega$, since $\beta_\epsilon$ is monotonic, we have
$$\nabla \varphi(x_0)=\nabla u^{\epsilon}(x_0), \ D^2 \varphi(x_0)\geq D^2 u^{\epsilon}(x_0).$$
Hence,
$$f(x_0, u^{\epsilon}(x_0)) =\{H_k[u^{\epsilon}](x_0)\}^{\frac{1}{k}}
\leq \{H_k[\varphi](x_0)\}^{\frac{1}{k}}.$$
Therefore,
\beqs
f(x, u^\epsilon)\leq C_0,
\eeqs
where $C_0>0$ is independent of $\epsilon$.
This uniform bound implies
\beq
\sup_{\Omega}(u^{\epsilon}-\varphi)\to 0
\eeq
as $\epsilon\to 0$.

By Corollary 3.4, there exists $v_0\in C^{0,1}(\overline\Omega)\cap \mathcal{SH}_k(\Omega)$,
\beq
H_k[v_0]=0, \   v_0|_{\partial\Omega}=\varphi-\delta.
\eeq
Hence by the comparison principle, we have $u^\epsilon\leq v_0$.
Let $E=\{u^\epsilon<\varphi-\frac{\delta}{2}\}$
and  $E_0=\{v_0<\varphi-\frac{\delta}{2}\}$.
Then $E\supset E_0$ and
$E_0$ is a neighbourhood of $\partial \Omega$, independent of $\epsilon$.
Therefore, by the argument in \cite {I2} for the boundary gradient estimate,
and by our construction of $\beta_\epsilon$, we obtain
\beq
\sup_{\partial\Omega} |\nabla u^\epsilon|\leq C_1 ,
\eeq
where $C_1$ is independent of $\epsilon$.

To obtain the global estimate, we denote for simplicity that $u=u^\epsilon$.
Following \cite {CNS} we introduce the auxiliary function
$$G=|\nabla u| e^{u}.$$
Suppose that the maximum of $G$ is attained at $x_0$.
If $x_0\in\partial\Omega$, the gradient estimates follows by (6.15).
If $x_0\in\Omega$, we may suppose that $|\nabla u|=u_1$
and $u_j=0$ for $j\geq 2$ at $x_0$.
Then at $x_0$,
\beqn
0&=&(\log G)_i=\frac{u_{1i}}{u}+u_i,\\
0&\geq& (\log G)_{ij}=\frac{u_{1ij}}{u}-\frac{u_{1i}u_{1j}}{u^2}+u_{ij}.
\eeqn
We now make use of the matrix
\beq
a_{ij}=\frac{1}{w}\left[u_{ij}-\frac{u_iu_ku_{kj}}{w(1+w)}-\frac{u_ju_ku_{ki}}{w(1+w)}
+\frac{u_iu_ju_ku_lu_{kl}}{w^2(1+w)^2}\right],
\eeq
where $w=\sqrt{1+|\nabla u|^2}$.
Let $F_k=\sigma_k(\lambda)$, where $\lambda\in\mathbb R^n$ denotes the eigenvalue vector
of the matrix $(a_{ij})$, and $\sigma_k$ denotes $k^{th}$ elementary symmetric function of $\lambda$.
It is known \cite{CNS} that
$$F_k(a_{ij})=H_k[u]=f(x, u)=-\beta_\epsilon(\varphi-u).$$
By (6.16), $u_{1i}=0$ for $i\geq 2$. By a rotation of the coordinates $(x_2, ..., x_n)$, we may further
assume that $u_{ij}$ is diagonal at $x_0$. Then $a_{ij}$ is also diagonal and
\beq
a_{11}=\frac{u_{11}}{w^3}, \  \ a_{ii}=\frac{u_{ii}}{w}, i>1.
\eeq
Set $F^{ij}=\frac{\partial F_k}{\partial a_{ij}}$. $F^{ij}$ is diagonal at $x_0$.
Assume that $u_1(x_0)>\varphi_1(x_0)$ otherwise we immediately get the gradient estimate.
Differentiating the equation with respect to $x_1$, we have
\beq
F^{ii}\frac{\partial a_{ii}}{\partial x_1}=\beta_\epsilon'(\varphi-u)\cdot (u_1-\varphi_1)\geq 0.
\eeq
By computation,
\beq
\frac{\partial a_{11}}{\partial x_1}=\frac{u_{111}}{w^3}-\frac{3u_{11}^2}{w^5},
\ \ \frac{\partial a_{ii}}{\partial x_1}=-\frac{u_1u_{11}}{w^2}a_{ii}+\frac{u_{ii1}}{w}, i>1.
\eeq
Combining (6.19) - (6.21), we have
\beq
\frac{2u_1F^{11}u_{11}^2}{w^5}+\frac{u_1u_{11}}{w^2}\sum_{i=1}^n F^{ii}a_{ii}
\leq \frac{1}{w}\sum_{i>1} F^{ii}u_{ii1}+\frac{1}{w^3}F^{11}u_{111}.
\eeq
By (6.17),
\beqs
\frac{1}{w^3}F^{11}u_{111}\leq \frac{F^{11}}{w^3}(\frac{u_{11}^2}{u_1}-u_1u_{11})
=\frac{F^{11}u_{11}^2}{u_1w^3}-u_1F^{11}a_{11}
\eeqs
and
\beqs
\frac{1}{w}\sum_{i>1}F^{ii}u_{ii1}\leq -\frac{1}{w}\sum_{i>1}F^{ii}u_1u_{ii}
=-u_1\sum_{i>1}F^{ii}a_{ii}.
\eeqs
Substituting them into (6.22) and using (6.16), we have
$$\frac{F^{11}u_{11}^2}{u_1w^5}(w^2-2)+\frac{u_1}{w^2}\sum_{i=1}^nF^{ii}a_{ii}\leq 0.$$
Since $u$ is smooth and $H_k$-subharmonic, $F^{11}>0$ and $\sum_{i=1}^nF^{ii}a_{ii}>0$.
Therefore, we have $w^2-2\leq 0$. In conclusion, there exists $C_1'$
depending only on $\sup_\Omega |u^\epsilon|$
and $\sup_{\partial\Omega}|\nabla u^\epsilon| $ and $\varphi$, but independent of $\epsilon$,
such that
$\sup_\Omega |\nabla u^\epsilon|\leq C_1'.$

\vskip 10pt

Note that the maximum bound of $|u^\epsilon|$ implies that
$$f(x, u^\epsilon)=-\beta_\epsilon(\varphi-u^\epsilon)\geq \nu_\epsilon>0$$
 for some positive constant $\nu_\epsilon$.
Hence we have the second derivative estimate \cite {I2}
\beq
\sup_\Omega |D^2 u^\epsilon|\leq C.
\eeq
By the global regularity theory of Evans-Krylov \cite{E, Kr}, we then have
\beq
\|u^\epsilon\|_{C^{3,\alpha}(\overline \Omega)}\leq C,
\eeq
where $C$ depending on
$\beta_\epsilon$, $\varphi$, $\Omega$.

\vskip10pt

To conclude the proof, we need to
show that $u^{\epsilon}$ converges to $\tilde \varphi^\delta$ and
$\tilde \varphi^\delta$ is  Lipchitz continuous.
We have shown that
$\nabla u^\epsilon$ are uniformly bounded with respect to $\epsilon$.
Hence by choosing a subsequence, $u^{\epsilon}$ converges to
a Lipchitz continuous, $H_k$-subharmonic function $\bar u$.
We claim that $\bar u=\tilde \varphi^\delta$ in $\Omega$.
Indeed, by (6.13), $\bar u\leq \varphi$ in $\Omega$, i.e, $\bar u\in S_{\varphi,\delta}$.
This implies $\bar u\leq \tilde \varphi^\delta$ in $\Omega$.
Through the definition of $\beta_\epsilon$, we also have
$H_k[\bar u]=0$ in viscosity sense in the set $G=\{\bar u<\varphi\}\subset \Omega$.
If there is $x_0\in\Omega$
such that $\bar u(x_0) <\tilde \varphi^\delta(x_0)$, then by the definition of $\tilde \varphi^\delta$,
there exists $h\in \mathcal{SH}_k(\Omega)\cap C^0(\Omega)$, such that
$\bar u(x_0)<h(x_0)$.  The set $G'=\{\bar u<h\}\subset G$ is nonempty.
Since
\beqs
 H_k[h]&\geq & H_k[\bar u]=0 \text{\ in $G'$},\\
  \bar u&=&h \text{\ on \ } \partial G',
\eeqs
by comparison, we obtain the contradiction.
Hence $\bar u=\tilde \varphi^\delta$. The theorem is proved.

\begin{rem}
We point out that the idea above can be used to prove
the approximation results for other nonlinear elliptic equations.
One example is the complex Monge-Amp\`ere equation on K\"ahler manifolds
in the next section.
\end{rem}

\section{Appendix 2: Regularization of plurisubharmonic functions on compact K\"ahler manifolds}

In this appendix, we give a new proof for the global regularization
of plurisubharmonic functions
on compact K\"ahler manifolds following the idea in last appendix.

The local regularisation of plurisubharmonic functions on domains in $\mathbb C^n$
is not difficult since the standard mollification keeps the plurisubharmonicity.
An interesting question is whether one can globally regularize plurisubharmonic functions.
The global regularisation of plurisubharmonic functions on $\mathbb C^n$ was obtained
in \cite{CS}. Since then, the problem on compact K\"ahler manifolds attracted much attention.

Let $(M, \omega)$ be a compact K\"ahler manifold.
A function $u:M\longrightarrow [-\infty, \infty)$ is called
$\omega$-plurisubharmonic (quasi-plurisubharmonic), if
$u$ is an integrable, upper semicontinuous function on $M$ and
$\omega_u=\omega+\partial\bar\partial u\geq 0$ in the distribution sense.
Denote by $PSH(M,\omega)$ the set of all  $\omega$-plurisubharmonic functions.
The problem of global regularisation refers to  whether an
$\omega$-plurisubharmonic function can be approximated by
smooth $\omega$-plurisubharmonic functions.
This approximation was first obtained in \cite{D, DP, GZ1}
under the assumption that $M$ admits a positive holomorphic line bundle,
i.e. $\omega$ is a Hodge class.
The proof used a complicated method developed by Demailly.
Later on, a simpler proof was given \cite{BK},
where the approximation holds in arbitrary compact K\"ahler case.
Here we present a new proof of this theorem here by solving penalized complex
Monge-Amp\`ere equations.

\begin{theo}
 Let $(M, \omega)$ be a compact K\"ahler manifold and $u\in PSH(M,\omega)$.
There exists a sequence of smooth $\omega$-plurisubharmonic functions $\{u_k\}$,
such that $u_k$ converges decreasingly to $u$ a.e..
\end{theo}

As in the previous section, the sequence $u_k$ is the solution
to an associated penalty problem, namely equation (7.3) below.
In Theorem 7.1, we  may assume that $u$ is bounded.
Otherwise, we may choose the sequence of
$\omega$-plurisubharmonic functions $u_j=\max\{u, -j\}$, $j=1, 2, 3,...$,
which converges to $u$ monotone decreasingly.
We also assume that $\omega_u\geq \sigma^\frac{1}{n}\omega$ for some $\sigma>0$.
For if not, we can replace $\omega_u$
by $\omega +(1- \sigma^\frac{1}{n})\sqrt{-1}\partial\bar\partial u$.
This implies that
\begin{equation}
(\omega+\sqrt{-1}\partial\bar\partial u)^n\geq \sigma\omega^n
\end{equation}
in viscosity sense.

The following theorem of Aubin and Yau {\cite{Au, Y}} will be needed.

\begin{theo}
Let $(M,\omega)$ be a compact K\"ahler manifold. Let $F$ be a smooth function defined
on $M\times \mathbb R$ with $\frac{\partial F}{\partial t}\geq 0$.
Suppose that, for some smooth function $\psi$ defined on $M$,
$\int_M e^{F(x,\psi(x))}\omega^n=vol(M)$.
Then there exists a smooth function $\varphi$ on $M$ such that
\begin{equation}
(\omega+\sqrt{-1}\partial\bar\partial \varphi)^n=e^{F(x, \varphi(x))}\omega^n,
\end{equation}
and $\omega+\sqrt{-1}\partial\bar\partial \varphi$ defines a K\"ahler metric.
Furthermore, any other smooth function satisfying the same property differs
from $\varphi$ by only a constant.
\end{theo}

Since $u$ is upper semicontinuous, there exist a sequence of smooth functions $\varphi_k$
which converge decreasingly to $u$. Furthermore, we may assume that
$\varphi_k\geq \varphi_{k+1}+\frac{1}{(k+1)^2}$.
For any integer $k$ and $\epsilon>0$, we consider the equation
\begin{equation}
(\omega+\sqrt{-1}\partial\bar\partial u)^n=\beta_{\epsilon}(\varphi_k-u)\omega^n,
\end{equation}
where  $\beta_\epsilon(t)$ is the smooth function given in Appendix 1.
It is easy to see that $\beta_{\epsilon}(0)=1$, so $\beta_{\epsilon}(\varphi_k-u)$ satisfies
condition (7.2) in Theorem 7.1.
For any $k$ and $\epsilon>0$, there exists a smooth solution
$u_{k,\epsilon}\in PSH(M,\omega)$ to the above equation (7.3).

We show that $u_{k,\epsilon}$ converges to $u$. First, we claim that
\begin{equation}
\sup_M\{u_{k,\epsilon}-\varphi_k,0\}\to 0
\end{equation}
as $\epsilon\to 0$. By the definition of $\beta_\epsilon$,
it suffices to show that $\beta_{\epsilon}(\varphi_k-u_{k,\epsilon})$
is uniformly bounded with respect to $\epsilon$.
Assume that $\beta_{\epsilon}(\varphi_k-u_{k,\epsilon})$ attains its maximum at $p\in M$.
Then $\sqrt{-1}\partial\bar\partial (\varphi_k-u_{k,\epsilon})\geq 0$. It follows that
\begin{equation}
\beta_{\epsilon}(\varphi_k-u_{k,\epsilon})=\frac{(\omega+\sqrt{-1}\partial\bar\partial u_{k,\epsilon})^n}{\omega^n}\leq
\frac{(\omega+\partial\bar\partial \varphi_{k})^n}{\omega^n}\leq C
\end{equation}
for some $C$ independent of $\epsilon$.
Next, we consider the lower bound of $u_{k, \epsilon}$.
For any $\delta>0$, let $\Omega_\delta=\{x\in M: u_{k,\epsilon}<\varphi_k-\delta\}$.
It is clear that
$$\beta_{\epsilon}(\varphi_k-u_{k,\epsilon})\leq e^{-\epsilon^{-1}\delta}<\sigma$$
as $\epsilon\to 0$. Then by (7.1), we have
\begin{eqnarray*}
&&u_{k,\epsilon}(x)\geq u(x)-\delta, \ x\in \partial \Omega_\delta\\
&&(\omega+\sqrt{-1}\partial\bar\partial u)^n> (\omega+\sqrt{-1}\partial\bar\partial u_{k,\epsilon})^n, \ x\in\Omega_\delta
\end{eqnarray*}
in viscosity sense. We claim that $u_{k,\epsilon}\geq u-\delta$ in $\Omega_\delta$.
By contradiction, suppose $m=\min(u_{k,\epsilon}-u+\delta)<0$ is attained at some point $p\in\Omega_\delta$.
Let $v=u_{k,\epsilon}-m$. It follows that $v(p)=u(p)$, $v\geq u$ in $\Omega_\delta$.
Hence, as a smooth test function $v$ satisfies
$$(\omega+\sqrt{-1}\partial\bar\partial v)^n\geq (\omega+\sqrt{-1}\partial\bar\partial u)^n$$
at $p$. This is a contradiction. Letting $\delta\to 0$, we obtain
$u_{k,\epsilon}\geq u$ on $M$. Along with (7.5), we have $u_{k,\epsilon}$ converges to $u$ as
$\epsilon\to 0$ and $k\to \infty$.
It is also clear that $\|u_{k,\epsilon}\|_{L^\infty}\leq C$ for some $C$ independent of $\epsilon$.

\vskip 10pt
In order to obtain a decreasing sequence,
we show that $u_{k, \epsilon}$ converges uniformly to a limit as $\epsilon\to 0$. We need
the following gradient estimate.

\begin{lem}
Let $-B$ be the lower bound for the bisectional curvature, where
$B\geq 0$. There exists $C>0$ depending on $\varphi_k$, $B$, and $\|u\|_{L^\infty}$,
but independent of $\epsilon$, such that
$$|\nabla u_{k,\epsilon}|\leq C.$$
\end{lem}

\begin{proof}
The proof is a generalisation of the gradient estimate by Blocki \cite{Bl}.
Denote $u=u_{k,\epsilon}$ and $f(x, u)=\beta_{\epsilon}(\varphi_k-u_{k,\epsilon})$. Following \cite{Bl},
we consider the auxiliary function
$$\alpha=\log|\nabla u|^2+\gamma(u),$$
where
$\gamma$ is a function of $u$.
Assume that $\alpha$ attains the maximum at some $p\in M$. Near $p$,
$\omega=\partial\bar\partial g$ for some smooth plurisubharmonic function $g$.
Write $v=g+u$. We can choose holomorphic charts $\{z_1, ..., z_n\}$ near $p$ such that
\begin{equation}
g_{i\bar j}=\delta_{i\bar j},\ g_{i\bar j l}=0
\end{equation}
and $(u_{i\bar j})$ is diagonal. By the computation in \cite{Bl} with the lower bound of
the bisectional curvature, it follows
\begin{eqnarray}
0\geq \sum_{l}\frac{\alpha_{l\bar l}}{v_{l\bar l}}&\geq &
(\gamma'-B)\sum_{l}\frac{1}{v_{l\bar l}}
+\frac{2}{|\nabla u|^2}Re\{\sum_{j}\frac{\partial \log f}{\partial z_j}u_{\bar j}\}
+\frac{1}{|\nabla u|^2}\sum_{l, j}\frac{|u_{jl}|^2}{v_{l\bar l}}
\nonumber\\
&&-[(\gamma')^2+\gamma'']\sum_l \frac{|u_{l}|^2}{v_{l\bar l}}-n\gamma'
\end{eqnarray}
It is shown in \cite{Bl} that
\begin{equation}
\frac{1}{|\nabla u|^2}\sum_{l, j}\frac{|u_{jl}|^2}{v_{l\bar l}}
\geq (\gamma')^2\sum_l \frac{|u_{l}|^2}{v_{l\bar l}}-2\gamma'-\frac{2}{|\nabla u|^2}.
\end{equation}
Assume that $|\nabla u|\geq C\geq 1$ is large, depending on $\varphi_k$,
so that by definition of $f$,
$$\frac{2}{|\nabla u|^2}Re\{\sum_{j}\frac{\partial \log f}{\partial z_j}u_{\bar j}\}
=\frac{2}{|\nabla u|^2}\frac{\beta'_\epsilon}{\beta_\epsilon}Re\{\sum_{j}\frac{\epsilon^{-1}\partial (\varphi_k-u)}{\partial z_j}u_{\bar j}\}\geq 0.$$
Then we have
\begin{equation}
0\geq (\gamma'-B)\sum_{l}\frac{1}{v_{l\bar l}}
-\gamma''\sum_l \frac{|u_{l}|^2}{v_{l\bar l}}-(n+2)\gamma'-2.
\end{equation}
Choose
$$\gamma(u)=(B+3)u-\frac{1}{\|u\|_{L^\infty}}u^2.$$
Then
\begin{equation}
B+1\leq\gamma'=B+3-\frac{2}{\|u\|_{L^\infty}}u\leq B+5,\ \gamma''=-\frac{2}{\|u\|_{L^\infty}}u.
\end{equation}
By (7.9),
$$\sum_{l}\frac{1}{v_{l\bar l}}+\frac{2}{\|u\|_{L^\infty}}\sum_l \frac{|u_{l}|^2}{v_{l\bar l}}\leq D:=(n+2)(B+5)+2.$$
Therefore $\frac{1}{v_{l\bar l}}\leq D$. Then $v_{l\bar l}\leq \sup f \cdot D^{n-1}$.
Using (7.10) again, we have
$$|\nabla u|^2\leq \frac{\|u\|_{L^\infty}\sup f \cdot D^{n}}{2}$$
at $p$. Combing this with the definition of $\alpha$, we finished the proof of the Lemma.
\end{proof}

By the above lemma, $u_{k, \epsilon}$ converges uniformly to
a limit $\tilde \varphi_k$ as $\epsilon\to 0$.
To get a decreasing convergent sequence, it remains to prove that
$\tilde\varphi_k\geq \tilde\varphi_{k+1}+\frac{1}{(k+1)^2}$.
We claim that $u_{k, \epsilon}\geq  u_{k+1, \epsilon}+\frac{1}{(k+1)^2}$.
Otherwise, $E_{k,\epsilon}=\{u_{k, \epsilon}< u_{k+1, \epsilon}+\frac{1}{(k+1)^2}\}$ is an open set.
Then $\varphi_k-u_{k,\epsilon}\geq \varphi_{k+1}-u_{k+1, \epsilon}$, i.e.,
$$(\omega+\sqrt{-1}\partial\bar\partial u_{k,\epsilon})^n\leq
(\omega+\sqrt{-1}\partial\bar\partial u_{k+1, \epsilon})^n$$
in $E_k$. The contradiction follows from $u_{k, \epsilon}=u_{k+1, \epsilon}$
on $\partial E_{k,\epsilon}$.
Letting $\epsilon\to 0$,
we obtain $\tilde\varphi_k\geq \tilde\varphi_{k+1}+\frac{1}{(k+1)^2}$,
so the sequence $\{u_{k,\epsilon}\}$ is monotone decreasing.


\begin{thebibliography}{999}


\bibitem{A} Ash, R. B.,
Measure, Integration and Functional Analysis, Academic Press, 1972.

\bibitem{Au} Aubin, T.,
\'Equations du type Monge-Amp\`ere surles vari\'et\'es k\"hleriennes compactes.
C. R. Acad. Sci. Paris S�r. 283(1976), 119-121.

\bibitem{BEGZ}Boucksom, S., Eyssidieux, P., Guedj, V. and Zeriahi, A,
Monge-Amp\`ere equations in big cohomology classes,
Acta Math. 205(2010), 199-262.

\bibitem{Bl} Blocki, Z.,
A gradient estimate in the Calabi-Yau theorem,
Math. Ann. 344(2009), 317-327.

\bibitem {BT1} Bedford, E. and Taylor, B.A.,
The Dirichlet problem for a complex Monge-Amp\`ere equation,
Invent. Math. 37(1976), 1-44.

\bibitem {BT2} Bedford, E. and Taylor, B.A.,
A new capacity for plurisubharmonic functions,  Acta Math. 149 (1982),  1-40.

\bibitem{BK} Blocki, Z. and Kolodziej, S.,
On regularization of plurisubharmonic functions on manifolds,
Proc. Amer. Math. Soc. 135 (2007), no. 7, 2089-2093

\bibitem {CNS} Caffarelli, L., Nirenberg, L. and Spruck, J.,
The Dirichlet problem for nonlinear second-order elliptic equations. V.
the Dirichlet problem for Weingarten hypersurfaces,
Comm. Pure. Appl. Math. (1) 41(1988), 47-70.
		
\bibitem {Ce1} Cegrell, U.,
Discontinuity of the complex Monge-Amp\`ere operator,
C. R. Acad. Sci. Paris Ser. I Math. 296 (1983),   869--871.

\bibitem {Ce2} Cegrell, U.,
Pluricomplex energy, Acta Math. 180 (1998),  187--217.

\bibitem {CCD} Chazal, F., Cohen-Steiner, D. and Merigot, Q.,
Boundary measures for geometric inference,
Found. Comput. Math. 10 (2010),   221--240.

\bibitem {CCL} Chazal, F., Cohen-Steiner, D., Lieutier, A. and Thibert, B.,
Stability of Curvature Measures,
Computer Graphics Forum, 28 (2009), 1485--1496.

\bibitem {CH}  Colesanti, A. and Hug, D.,
Steiner type formulae and weighted measures of singularities for semi-convex functions, Trans. Amer. Math. Soc. 352 (2000),  3239--3263.

\bibitem {CM}  Colesanti, A. and Manselli, P.,
Geometric and isoperimetric properties of sets of positive reach in $E^d$,
Atti Semin. Mat. Fis. Univ., 57 (2010), 97--113.

\bibitem{CS} Cegrell, U. and Sadullaev, A.,
Approximation of plurisubharmonic functions and Dirichlet problem
for the complex Monge-Amp\`ere operator,
Math. Scand. 71(1992), 62-68.


\bibitem {DTW} Dai, Q.Y., Trudinger, N.S. and Wang, X.-J.,
The mean curvature measure,
Jour. Euro. Math. Soc. (3) 14(2012), 779-800.

\bibitem{D} Demailly, J. P.,
Regularization of closed positive currents and intersection theory,
J. Alg. Geom. 1 (1992), 361-409.

\bibitem{DP} Demailly, J. P. and Paun, M.,
Numerical characterization of the K\"ahler cone of a compact K\"ahler manifold,
Ann. Math. 159 (2004), 1247-1274.

\bibitem{E} Evans, L. C,
Classical solutions of fully nonlinear, convex, second-order elliptic equations, Comm.Pure Appl. Math.35(1982), 333–363.

\bibitem {F} Federer, H.,
Curvature measures,
Trans. Amer. Math. Soc. 93(1959), 418-491.

\bibitem {G1} Giusti, E.,
       On the equation of surfaces of prescribed mean curvature,
       Existence and uniqueness without boundary conditions,
       Invent. Math. 46(1978), 111-137.

\bibitem {G2} Giusti, E.,
       Generalized solutions for the mean curvature equation,
       Pacific J. Math., 88(1980), 297-321.

\bibitem{GZ1} Guedj, V. and Zeriahi, A.,
Intrinsic capacities on compact K\"ahler manifolds,
J. Geom. Anal.15 (2005), 607-639.

\bibitem {HKM} Heinonen, J., Kilpelainen T. and Martio, O.,
Nonlinear potential theory of degenerate elliptic equations, Oxford Univ. Press, 1993.

\bibitem {HGW} Hug, D., Last, G. and Weil, W.,
A local Steiner-type formula for general closed sets and applications,
Math. Z. 246 (2004),   237--272.

\bibitem {I1} Ivochkina, N.,
Solutions of the Dirichlet problem for equations of mth order curvature,
Math. USSR-Sb 67(1990), 317-339.

\bibitem {I2} Ivochkina, N.,
The Dirichlet problem for the curvature equation of order $m$,
Leningrad Math. J. (3) 2(1991), 631-654.

\bibitem {Kl} Klimek, M.,
Pluripotential Theory, Oxford University Press, New York, 1991.

\bibitem {Kol} Kolodziej, S.,
The complex  Monge-Amp\`ere equation,
Acta Math. 180(1998), 69-117.

\bibitem{Kr} Krylov, N.V.,
Boundedly inhomogeneous elliptic and parabolic equations,
Izv. Akad. Nauk SSSR Ser. Mat. 46(1982),487–523.

\bibitem {K} Korevaar, N.,
A priori interior gradient bounds for solutions to elliptic Weingarten equations,
Ann. Inst. H. Poincare Anal. Non Lineaire 4(1987), 405-421.

\bibitem {L} Lee, K.,
The obstacle problem for Monge-Amp\`ere equation,
Comm. Part. Diff. Eqn. 26(2001), 33-42.

\bibitem {Re} Reilly, R.C.,
On the Hessian of a function and curvatures of its graph,
Michigan Math. J. 20(1973),   373--383.

\bibitem {S} Schneider, R.,
Convex bodies: the Brunn-Minkowski Theory,
Cambridge Univ. Press, Cambridge(1993).

\bibitem {Ta1} Takimoto, K.,
Isolated singularities for some types of curvature equations,
J. Diff. Eqns. 197 (2004),  275--292.

\bibitem {Ta2} Takimoto, K.,
Singular sets for curvature equation of order k,
J. Math. Anal. Appl. 309 (2005),   227--237.

\bibitem {T1} Trudinger, N.S.,
The Dirichlet problem for the prescribed  curvature equations,
Arch. Rational. Mech. Anal. 111(1990), 153-179.

\bibitem {T2} Trudinger, N.S.,
A priori bounds and necessary conditions for solvability of prescribed curvature equations,
Manuscripta. Math. 67(1990), 99-112.

\bibitem {TW1} Trudinger, N.S. and Wang, X.-J.,
Hessian measures I,
Topol. Methods Nonlinear Anal. 10(1997), 225-239.

\bibitem {TW2} Trudinger, N.S. and Wang, X.-J.,
Hessian measures II,
Ann. Math. 150(1999), 579-604.

\bibitem {W1}  Wang, X.-J.,
Interior gradient estimates for mean curvature equations,
Math. Z. 228 (1998), 73-81.

\bibitem {W2}  Wang, X.-J.,
The k-Hessian equation, in {\it Geometric analysis and PDEs},
 Lecture Notes in Math., vol. 1977, (2009) 177--252.

\bibitem{Y}
Yau, S.T.,
On the Ricci curvature of a compact K\"ahler manifold and
the complex Monge-Amp\`ere equation. I",
Comm. Pure. Appl. Math. 31(1978), 339-411.


\end{thebibliography}
\end{document}